\documentclass{amsart}

\usepackage{srcltx}
\usepackage{amsmath}
\usepackage{amssymb}
\usepackage{graphicx}
\usepackage{amsmath,amscd}
\usepackage{tikz}
\usepackage{tikz-cd}
\usetikzlibrary{positioning,arrows}
\usepackage[utf8]{inputenc}
\usepackage[T1]{fontenc}
\usepackage{hyperref}
\usepackage{enumitem}

\numberwithin{equation}{section}
\theoremstyle{plain}

\newtheorem{lem}[equation]{Lemma}

\newtheorem{prop}[equation]{Proposition}
\newtheorem{thm}[equation]{Theorem}
\newtheorem{cor}[equation]{Corollary}

\newtheorem{conjecture}[equation]{Conjecture}

\theoremstyle{definition}
\newtheorem{definition}[equation]{Definition}
\newtheorem{remark}[equation]{Remark}
\newtheorem{example}[equation]{Example}

\newcommand{\field}[1]{\mathbb{#1}}
\newcommand{\integers}{\ensuremath{\field{Z}}}

\newcommand{\naturals}{\ensuremath{\field{N}}}

\newcommand{\C}{\mathbb C}
\newcommand{\G}{\mathcal G}

\newcommand{\Aut}[1]{\text{Aut}(#1)}

\newcommand{\semidirect}{\rtimes}
\newcommand{\acts}{\curvearrowright}

\newcommand{\neb}{\ensuremath{\mathcal N}}

\let\diameter\relax \DeclareMathOperator{\diameter}{\text{diameter}}

\newcommand{\systole}[1]{\ensuremath{\|{#1}\|}}
\newcommand{\nclose}[1]{\ensuremath{\langle\!\langle#1\rangle\!\rangle}}

\usetikzlibrary{calc}
\newcommand*{\vertchar}[2][0pt]{%
  \tikz[
    inner sep=0pt,
    shorten >=-.15ex,
    shorten <=-.15ex,
    line cap=round,
    baseline=(c.base),
  ]\draw [line width=0.05mm,  black]
    (0,0) node (c) {#2}
    ($(c.south)+(#1,0)$) -- ($(c.north)+(#1,0)$);%
}

\newcommand{\pstab}{\operatorname{\textup{\vertchar{S}tab}}}
\newcommand{\stab}{\operatorname{Stab}}

\newcommand{\showcomments}{yes}
\renewcommand{\showcomments}{no}

\newsavebox{\commentbox}
{\ifthenelse{\equal{\showcomments}{yes}}%
{\footnotemark
        \begin{lrbox}{\commentbox}
        \begin{minipage}[t]{1.25in}\raggedright\sffamily\tiny
        \footnotemark[\arabic{footnote}]}
{\begin{lrbox}{\commentbox}}}%
{\ifthenelse{\equal{\showcomments}{yes}}%
{\end{minipage}\end{lrbox}\marginpar{\usebox{\commentbox}}}
{\end{lrbox}}}

\begin{document}
\title{Stature and separability in graphs of groups}

\begin{abstract}
We introduce the notion of \emph{finite stature} of a family $\{H_i\}$ of subgroups  of a group $G$.
We investigate the separability of subgroups of a group $G$ that splits
as a graph of hyperbolic special groups with quasiconvex edge groups.
We prove that when the vertex groups of $G$ have finite stature,
then quasiconvex subgroups of the vertex groups of $G$ are separable in $G$.
We present some partial results in a relatively hyperbolic framework.
\end{abstract}

\author{Jingyin Huang}
\address{Dept. of Math. \& Stats.\\
	McGill University \\
	Montreal, Quebec, Canada}
\curraddr{Department of Mathematics, The Ohio State University, 231 W. 18th Ave, Columbus, Ohio, U.S. 43210}
\email{huang.929@osu.edu}

\author{Daniel T. Wise}
\address{Dept. of Math. \& Stats.\\
	McGill University \\
	Montreal, Quebec, Canada}
\email{daniel.wise@mcgill.ca}

\maketitle
	
	\tableofcontents
\section{Introduction}	

\subsection{Stature and Statement of Main Result}
\begin{definition}[Stature]
	\label{def:finite stature}
	Let $G$ be a group and let $\{H_\lambda\}_{_{\lambda\in\Lambda}}$ be a collection of subgroups of $G$. Then $(G,\{H_\lambda\}_{_{\lambda\in\Lambda}})$ has \emph{finite stature} if for each $\mu\in \Lambda$, there are finitely many $H_{\mu}$-conjugacy classes of infinite subgroups of form $H_{\mu}\cap C$, where $C$ is an intersection of (possibly infinitely many) $G$-conjugates of elements of $\{H_\lambda\}_{\lambda\in\Lambda}$.
\end{definition}

We are especially interested in the finite stature of the vertex
groups in a splitting of $G$
as a graph of groups. An attractive equivalent definition of finite stature in that case
 is described in Definition~\ref{def:finite stature1}. 
Finite stature does not always hold for the vertex groups of a splitting,
but it is easy to verify in the following simple case:
 \begin{example}
If $G$ is a graph of groups such that each attaching map of each edge group is 
an isomorphism, then $G$ has finite stature with respect to its vertex groups. Note that this includes many nilpotent groups and solvable groups.
In Proposition~\ref{prop:NbyF dichotomy},
we generalize this to provide a characterization of finite stature
of the vertex groups when all edge groups in the splitting of $G$ are commensurable.
This can be applied to tree $\times$ tree lattices to see that irreduciblility corresponds to 
infinite stature of the vertex groups in the action on either factor.
\end{example}

The main result in this paper which is proven as Theorem~\ref{thm:main1} is the following:
\begin{thm} \label{thm:main10}
		Let $G$ be the fundamental group of a graph of groups with finite underlying graph. Let $\mathcal{V}$ be the collection of vertex groups of $G$. Suppose the following conditions hold:
		\begin{enumerate}
			\item each vertex group is word-hyperbolic and virtually compact special;
			\item each edge group is quasiconvex in its vertex groups;
			\item $(G,\mathcal{V})$ has finite stature.
		\end{enumerate}
		Then each quasiconvex subgroup of a vertex group of $G$ is separable in $G$. In particular, $G$ is residually finite.
\end{thm}

Though the vertex groups and edge groups of $G$ are hyperbolic, $G$ does not have to be (relatively) hyperbolic. For example, any free-by-cyclic group satisfies the assumption of  Theorem~\ref{thm:main10}, but not all such groups are relatively hyperbolic.

\subsection{Height}
We now recall the notion of height, which will sometimes play a facilitating role in 
determining finite stature.

\begin{definition}[Height]
	\label{def:height}
A subgroup $H\le G$ has \emph{finite height} if there does not exist
a sequence $\{g_n\}_{n\in \naturals}$ with $g_iH\neq g_j H$ for $i\neq j$ and 
with
$\cap_{n\in N} H^{g_n}$  infinite.
We use the  notation $H^g=gHg^{-1}$.

A finite collection $\mathcal H = \{H_1,\ldots, H_r\}$ of subgroups has \emph{finite height} if each $H_k$ has finite height.
\end{definition}

The above notion  was introduced in \cite{GMRS98}, where the \emph{height} $h$ is  the number $0\leq h \leq \infty$ 
that is the supremal length of  a sequence with infinite intersection of conjugates.
Particular attention was paid to word-hyperbolic groups,
and it was shown in \cite{GMRS98} that a quasiconvex subgroup of a word-hyperbolic group has finite height.
Since then, this notion has been studied in other contexts, in the relatively hyperbolic setting \cite{HruskaWisePacking},
and most recently in a graded relatively hyperbolic setting \cite{DahmaniMj2016}.

As it examines infinite intersections of conjugates,
finite stature is certainly related in various ways to finite height.
And we will utilize finite height of various subgroups
 to prove finite stature for the edge groups in certain graphs of groups.
The following crisply distinguishes the two notions:

\begin{example}[Height/Stature when abelian]
Let $G$ be abelian. Then $(G,\Lambda)$ has finite stature for any finite collection $\Lambda$.
 However, $(G,\Lambda)$ has infinite height precisely when $\Lambda$
has an element that is both infinite and infinite index in $G$.
\end{example}

\begin{example}
Let $M$ be a 3-manifold,
and consider the splitting of $\pi_1M$ arising from its JSJ splitting \cite{JacoShalen79}.
Then $\pi_1M$ has finite stature relative to its vertex groups.
Indeed, the intersection between adjacent vertex groups corresponds to the torus between them.
And the intersection between adjacent edge groups is the normal cyclic subgroup associated to the Siefert fibration of their common  vertex group. Finally, these normal cyclic subgroups intersect trivially
for consecutive vertex groups. We conclude that the only multiple intersection in a vertex group
arise from its tori, these normal cyclic subgroups, and the trivial group.
Finally, we caution that this example is deceptive since its controlling feature is that length~3 lines
 in the Bass-Serre tree have trivial stabilizer.
\end{example}

\subsection{Intuitive explanation of finite stature:}
As is clear from Theorem~\ref{thm:main10},
we are particularly interested in the finite stature of the collection of vertex 
groups in a splitting of $G$ as a graph of groups.
As the explanations in the text are more group theoretical, let us give a topological description  of finite stature in this setting.
Let $X$ be a graph of spaces associated to $G$ with underlying graph $\Gamma_X=\Gamma_G$,
and where each vertex space $X_v$ has $\pi_1X_v=G_v$ and each edge space $X_e$ has $\pi_1X_e=G_e$, and the attaching maps are $\pi_1$-injections. 
We are interested in \emph{transections} which are defined in Definition~\ref{def:big-tree}, but are roughly``maximal product regions'' in $X$ subject to having an underlying graph
being isomorphic to a specific tree. More precisely, we are interested in immersions of graphs of spaces $Y\rightarrow X$ 
where the underlying  graph $\Gamma_Y$ is a  tree,
and where all attaching maps are isomorphisms of groups. The reader should regard $Y$ as a product $Y_v \times \Gamma_Y$ where $Y_v$ is a vertex space of $Y$. We refer to $Y_v$ as the ``cross-section'' of $Y$. Maximality of the cross-section means that $Y\rightarrow X$ doesn't extend to an immersion $Y'\rightarrow X$ with $\Gamma_{Y'}=\Gamma_Y$ such that
the cross-section $Y_v'$ is bigger than $Y_v$ in the sense that $\pi_1Y_v\rightarrow \pi_1Y_v'$ is a proper inclusion. That $Y\rightarrow X$ is an ``immersion'' of graphs of spaces,
 means that the induced map $\widetilde Y\rightarrow \widetilde X$ induces
 an embedding between underlying Bass-Serre trees $\Gamma_{\widetilde Y}\rightarrow \Gamma_{\widetilde X}$. 
 finite stature means that for each vertex space $X_v$ of $X$,
 there are finitely many distinct cross-sections mapping into $X_v$ up to homotopy.
We refer to Example~\ref{exmp:Multiple HNN} and Figure~\ref{fig:CrossSections}.
We ignore the cross-sections with $\pi_1Y_v$ finite.
Sometimes the underlying tree $\Gamma_Y$ is infinite,
 so $Y$  threatens to intersect $X_v$ in infinitely many cross-sections, yielding infinite stature. However, in many cases, the map $Y\rightarrow X$
factors as $Y\rightarrow \bar Y \rightarrow X$ where $Y\rightarrow \bar Y$ is a covering space
and $\bar Y$ has a finite underlying graph, in which case $Y$ yields finitely many cross-sections in each vertex space  as desired.

\begin{figure}\centering
\includegraphics[width=.85\textwidth]{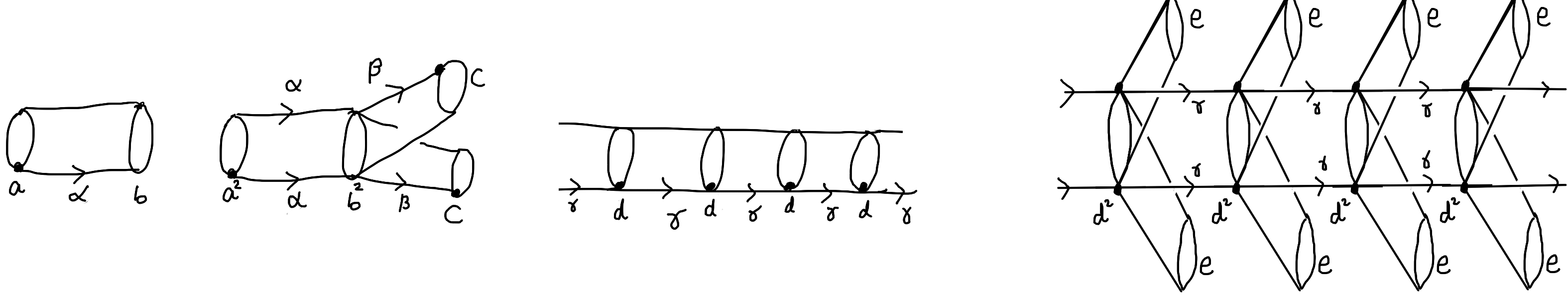}
\caption{Four products yielding 8 distinct cross-sections in the vertex space of Example~\ref{exmp:Multiple HNN}.
\label{fig:CrossSections}}
\end{figure}
\begin{example}\label{exmp:Multiple HNN}
Consider the following simplistic example of a multiple HNN extension of the free group on $a,b,c,d,e$
with stable letters $\alpha,\beta,\gamma,\delta$.
$$\langle a,b,c,d,e, \ \alpha, \beta, \gamma, \delta, \mid
a^\alpha=b,
(b^2)^\beta=c,
d^\gamma=d,
(d^2)^\delta=e\rangle$$
The various transections   are illustrated in Figure~\ref{fig:CrossSections}.
Each transection can contribute several cross-sections within the vertex space.
\end{example}

In the case of the HNN extension
$\langle a,t\mid a^t=aa\rangle$,
there is a single immersed product region whose underlying graph is an infinite 3-valent tree.
There are then infinitely many cross-sections in the vertex space, each corresponding to a circle $a^{2^n}$. Hence the group has infinite stature with respect to its vertex subgroup.
See Example~\ref{exmp:height:BS}.

\begin{example}[Tubular Groups]
A \emph{tubular group} $G$ splits as a finite graph $\Gamma$ of groups with $\integers^2$ for each vertex group and $\integers$ for each edge group.
We claim that $(G,\{G_v\})$ has infinite stature precisely when there is an embedding $BS(n,m)\hookrightarrow G$ with $n\neq \pm m$. Here $BS(n,m)=\langle a,t\mid t^{-1}a^nt=a^m\rangle$.
To see this,  consider an associated finite graph of groups $\Gamma'$
(that might be disconnected)  formed as follows:
Each edge $e$ of $\Gamma$ yields an edge of $\Gamma'$ with $G'_e=G_e$.
There is a vertex $u$ of $\Gamma'$ for each maximal cyclic subgroup of $G_v$ that is commensurable with an edge group at $G_v$, and we let $G'_u$ be this maximal cyclic subgroup. The edges of $\Gamma'$ are attached to the vertices according to the inclusion 
of the edge groups.
 Note that a single vertex of $\Gamma$ can contribute multiple vertices of $\Gamma'$. 
We refer to Figure~\ref{fig:TubularAssociate}
There is a map $\Gamma'\rightarrow \Gamma$, that induces maps between graphs of spaces.

Every transection for $G$ induces a transection in $G'$ and vice-versa.
As $G'$ is a graph of cyclic groups it is easy to see that it has finite stature if and only if it has no $BS(n,m)$ subgroup with $n\neq \pm m$ (see Lemma~\ref{prop:NbyF dichotomy}.)

Finite stature of a tubular group does not allow us to prove
residual finiteness, since there isn't an adequate version of the Malnormal Special Quotient Theorem for abelian groups. 
\end{example}

\begin{figure}\centering
\includegraphics[width=.9\textwidth]{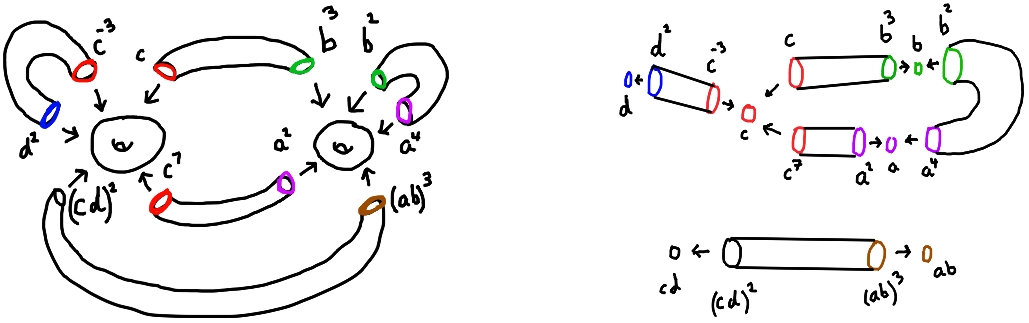}
\caption{The graph of spaces on the left, has an associated graph of spaces  carrying
all its transections on the right.
\label{fig:TubularAssociate}}
\end{figure}

\subsection{Connection to virtual specialness}
In \cite{HuangWiseSpecial} we formulate a conjecture relating the notion of  finite stature to virtual specialness of a compact nonpositively curved cube complex.
Specifically, we show that $(\pi_1X,\{\pi_1U_i\})$ has finite stature when
$X$ is a compact virtually special cube complex, and $\{U_i\}$ varies over the hyperplanes.
We conjecture that the converse holds.
We apply Theorem~\ref{thm:main10}  there
as a primary ingredient to prove the virtual specialness of certain nonpositively curved cube complexes. 
We emphasize that  Theorem~\ref{thm:main10} applies to prove the residual finiteness of many groups that aren't special.
Moreover, the relatively hyperbolic variant given in the following application
applies to groups that are not hyperbolic relative to virtually abelian subgroups,
since the parabolics can be $\integers^n\semidirect_\phi \integers$ where $\phi$ is an infinite order automorphism of $\integers^n$.

\begin{thm}
	\label{thm:rel hyperbolic intro}
Let $G$ be hyperbolic relative to subgroups that are virtually f.g.\ free abelian by $\mathbb Z$. Suppose $G$ splits as a finite graph of groups whose edge groups are relative quasiconvex and whose vertex groups are virtually sparse special. Then each relatively quasiconvex subgroup of each vertex group of $G$ is separable. In particular, $G$ is residually finite.
\end{thm}

Theorem~\ref{thm:rel hyperbolic intro} which is proven as Theorem~\ref	{thm:rel hyperbolic}
plays a role in the inductive proof of virtual specialness of certain relatively hyperbolic groups with quasiconvex hierarchies \cite[Sec~15]{WiseIsraelHierarchy}.

\section{Preliminaries}

\subsection{Background on Cube Complexes}\label{sub:cubical background}

\subsubsection{Nonpositively curved cube complexes:}
An \emph{$n$-dimensional cube} is a copy of $[-\frac12,+\frac12]^n$.
Its \emph{subcubes} are the subspaces obtained by restricting some coordinates to $\pm\frac12$.
We regard a subcube as a copy of a cube in the obvious fashion.
A \emph{cube complex} $X$ is a cell complex obtained by gluing cubes together along subcubes,
where all gluing maps are modeled on isometries.
Recall that a \emph{flag complex} is a simplicial complex with the property that a finite set of vertices spans a simplex if and only if they are pairwise adjacent.
$X$ is \emph{nonpositively curved} if the link of each $0$-cube of $X$ is a flag complex.
A \emph{CAT(0) cube complex} $\widetilde X$ is a simply-connected nonpositively curved cube complex.

\subsubsection{Hyperplanes}
A \emph{midcube} is a subspace of an $n$-cube obtained by restricting one coordinate of $[-\frac12,+\frac12]^n$ to $0$.  A \emph{hyperplane} $\widetilde U$ is  connected subspace of a CAT(0) cube complex $\widetilde X$
such that for each cube $c$ of $\widetilde X$, either  $\widetilde U\cap c =\emptyset$ or $\widetilde U\cap c$ consists of a midcube of $c$. The \emph{carrier} of a hyperplane $U$ is the subcomplex $N(\widetilde U)$ consisting of all closed cubes intersecting $U$.
We note that every midcube of $\widetilde X$ lies in a unique hyperplane, and $N(\widetilde U)\cong \widetilde U\times c^1$ where $c^1$ is a 1-cube. An \emph{immersed hyperplane} $U\rightarrow X$ in a nonpositively curved cube complex is a map $\stab(\widetilde U)\backslash \widetilde U \rightarrow X$ where $\widetilde U$ is a hyperplane of the universal cover $\widetilde X$ of $X$. We similarly define $N(U)\rightarrow X$ via $N(U)=\stab(\widetilde U)\backslash N(\widetilde U)$.

\newcommand{\link}{\text{link}}

A map $\phi:Y\rightarrow X$ between nonpositively curved cube complexes is \emph{combinatorial} if it maps open $n$-cubes homeomorphically to open $n$-cubes. A combinatorial map is a \emph{local-isometry}
if for each $0$-cube $y$,
the induced map $\link(y)\rightarrow \link(\phi(y))$ is an embedding of simplicial complexes,
such that $\link(y)\subset \link(\phi(y))$ is \emph{full} in the sense that if a collection of vertices
of $\link(y)$ span a simplex in $\link(\phi(y))$ then they span a simplex in $\link(y)$.

\subsubsection{Special Cube Complexes}
A nonpositively curved cube complex $X$ is \emph{special}
if each immersed hyperplane $U\rightarrow X$  is an embedding, and moreover $N(U)\cong U\times [-\frac12,+\frac12]$,
each restriction $U\times \{\pm\frac12\} \rightarrow X$ is an embedding,
and if $U,V$ are hyperplanes of $X$ that intersect then 0-cube of $N(U)\cap N(V)$ lies in a 2-cube
intersected by both $U$ and $V$.

\subsection{Cubical small cancellation}
A \emph{cubical presentation} $\langle X | \{Y_i\} \rangle$ consists of a nonpositively curved cube complex $X$, and a set of local isometries $Y_i \rightarrow X$ of nonpositively curved cube complexes. We use the notation $X^*$ for the cubical presentation above. As a topological space, $X^*$ consists of $X$ with a cone on each $Y_i$ attached,
so $\pi_1X^*=\pi_1X/\nclose{\{\pi_1Y_i\}}$. We use the notation $\widetilde X^*$ for the universal cover of $X^*$.

  $\systole{Y_i}$ denotes the infimal length of an essential combinatorial closed path in $Y_i$.
  
Let $Y_i\rightarrow X$ and $Y_j\rightarrow X$ be maps.
A \emph{morphism} $Y_i\rightarrow Y_j$ is a map such that $Y_i\rightarrow X$ factors as $Y_i\rightarrow Y_j\rightarrow X$.
It is an \emph{isomorphism} if there is an inverse map $Y_j\rightarrow Y_i$ that is also a morphism.
Define an automorphism accordingly and let $\Aut{Y\rightarrow X}$ denote the group of automorphisms of
 $Y\rightarrow X$.

A \emph{cone-piece} of $X^*$ in $Y_i$ is a component of $g\widetilde Y_j \cap \widetilde Y_i$ for some $g \in \pi_1X$, where we exclude the case that $i=j$ and $g\in \text{Stab}(\widetilde Y_i)$
and there is a map $\bar g: Y_i\rightarrow Y_i$ so that the following diagram commutes:
 $$\begin{array}{ccccc}
  \widetilde Y_i & \rightarrow  & Y_i           & \searrow &  \\
  g\downarrow      &              &  \bar g \downarrow   &          & X \\
  \widetilde Y_i  &  \rightarrow & Y_i           & \nearrow &  \\
\end{array}$$

For a hyperplane $\widetilde{U}$ of $\widetilde{X}$,  let $N(\widetilde{U})$ denote its \emph{carrier}, which is the union of all closed cubes intersecting $\widetilde{U}$.
A \emph{wall-piece} of $X^*$ in $Y_i$ is a component of $\widetilde{Y}_i \cap N(\widetilde{U})$, where $\widetilde{U}$ is a hyperplane that is disjoint from  $\widetilde{Y}_i$.

For instance, consider the presentation  $\langle a,b \mid (abbb)^{20}, (baaa)^{20} \rangle$,
and regard it as a cubical presentation $\langle X \mid Y_1, Y_2\rangle$ where $X$ is a bouquet of circles and each $Y_i$ is an immersed cycle.
Then the path $ab$ corresponds to a piece, since it appears as an intersection between distinct lines $\widetilde Y_1, \widetilde Y_2$ in $\widetilde X$.
Likewise $bb$ is a piece since it occurs as the intersection of two distinct translates of $\widetilde Y_1$. However, $bbb$ is not a piece, since any two translates of $\widetilde Y_1$ that contain $bbb$ are actually the same, and they differ by a translation that projects to an automorphism of $Y_1$.


\begin{definition}[Small Cancellation]\label{defn:small cancellation}
$X^*$ satisfies the $C'(\frac{1}{24})$ \emph{small cancellation} condition if
$\diameter(P) < \frac{1}{24} \systole{Y_i}$
for every cone-piece or wall-piece $P$ of  $Y_i$.
\end{definition}


\begin{definition}
Let $X^*=\langle X \mid \{Y_i\}\rangle$ and $A^*=\langle A \mid \{B_j\}\rangle$ be cubical presentations.
A \emph{map} $A^*\rightarrow X^*$ of cubical presentations is a local isometry $A\rightarrow X$,
so that for each $j$ there exists $i$ such that there is a map $B_j\rightarrow Y_i$
so that the composition $B_j\rightarrow A\rightarrow X$ equals $B_j \rightarrow Y_i\rightarrow X$.

Given a cubical presentation $X^*$ and a local isometry $A\rightarrow X$,
  the \emph{induced presentation} is the cubical presentation of the form
$A^* = \langle A \mid \{A\otimes_X Y_i\} \rangle$ where $A\otimes_X Y_i$ is the fiber-product of $A\rightarrow X$ and $Y_i\rightarrow X$. We refer to Section~\ref{subsubsec:fiber products} for the definition of fiber-product. It is immediate that there is a map of cubical presentations $A^*\rightarrow X^*$.
\end{definition}

The following is a slightly more restrictive version of the same notion treated in \cite[Def~3.61 \& Def~3.65]{WiseIsraelHierarchy}:
\begin{definition}
	Let $A^*=\langle A\mid \{B_j\}\rangle$ and $X^*=\langle X\mid\{Y_i\}\rangle$.
	We say $A^*\rightarrow X^*$ has \emph{liftable shells} provided the following holds:
	Whenever $QS\rightarrow Y_i$ is an essential closed path with $|Q| > |S|_{Y_i}$ and $Q\to Y_i$ factors through $Q\to A \otimes_X Y_i\to Y_i$,
	there exists $B_j$ and a lift $QS\rightarrow B_j$, such that $B_j\rightarrow A\rightarrow X$
	factors as $B_j\rightarrow Y_i\rightarrow X$.
\end{definition}

The following is a restatement of a combination of \cite[Thm~3.68 and Cor~3.72]{WiseIsraelHierarchy}:
\begin{lem}\label{lem:quasi-isom embedding}
Let  $X^*$ be $C'(\frac{1}{24})$.
Let $A^*\rightarrow X^*$ have liftable shells and suppose that $A^*$ is compact.
Then $\pi_1 A^*\to\pi_1 X^*$ is injective, $A^*\rightarrow X^*$ lifts to an embedding $\widetilde A^*\rightarrow \widetilde X^*$,
and moreover the map is a quasi-isometric embedding.
\end{lem}

The following appears as \cite[Lem~3.67]{WiseIsraelHierarchy}.

\begin{lem}\label{lem:liftable shell criterion}
Let $\langle X \mid \{Y_i\}\rangle$ be a $C'(\frac{1}{24})$ small-cancellation cubical presentation.
Let $A\rightarrow X$ be a local isometry and let $A^*$ be the associated induced presentation.
Suppose that for each $i$, each component of $A\otimes_X Y_i$ is either a copy of $Y_i$
or is a contractible  complex $K$ with $\diameter(K) \leq  \frac12\systole{Y_i}$.
Then the natural map $A^*\rightarrow X^*$ has liftable shells.
\end{lem}

\subsection{Special quotients}
\begin{definition}
	\label{def:almost malnormal}
	A collection of subgroups $\{H_1,\ldots,H_r\}$ of $G$ is \emph{malnormal} provided that $H^g_i\cap H_j=\{1_G\}$ unless $i=j$ and $g\in H_i$. Similarly, the collection is \emph{almost malnormal} if intersections of nontrivial conjugates are finite (instead of trivial). Note that this condition implies that $H_i\neq H_j$ (unless they are finite in the almost malnormal case).
\end{definition}

The following appears as \cite[Thm~12.2]{WiseIsraelHierarchy}:
\begin{thm}[Malnormal Virtually Special Quotient]\label{thm:malnormal special quotient}
Let $G$ be a word-hyperbolic group with a finite index subgroup $J$ that is the fundamental group
of a compact  special cube complex.
Let $\{H_1,\dots, H_r\}$ be an almost malnormal collection of quasiconvex subgroups of $G$.
Then there are finite index subgroups $\ddot H_1, \dots, \ddot H_r$ such that:
For any finite index subgroups $ H_1', \dots, H_r'$ contained in the $\ddot H_1, \dots, \ddot H_r$ the quotient:
$G'= G / \nclose{H_1', \dots, H_r'}$ is
a word-hyperbolic group with a finite index subgroup $J'$ that is the fundamental group of a compact special cube complex.
\end{thm}

The following is a simplified restatement of \cite[Lem~12.10]{WiseIsraelHierarchy}:
\begin{lem}\label{lem:Intersection Control}
Let $\langle X\mid Y_1,\dots, Y_k\rangle$ be a $C'(\frac{1}{24})$ small-cancellation cubical presentation.
Let $A_1\rightarrow X$ and $A_2\rightarrow X$ be based local isometries.
Suppose $X^*$ has small pieces relative to $A_1,A_2$ in the following sense
for each $1\leq j\leq 2$ and $1\leq i\leq k$:
For each pair of lifts $\widetilde A_j, \widetilde Y_i$ to $\widetilde X$,
the piece $P$ between $\widetilde A_j,\widetilde Y_i$ satisfies: Either $\diameter(P) < \frac18\systole{Y_i}$
or $\widetilde Y_i \subset \widetilde A_j$ and factor through a map $Y_i\to A_j$.

Let $\{\pi_1A_1g_i\pi_1A_2\}$ be a collection of distinct double cosets in $\pi_1X$.
And suppose that for each chosen representative $g_i$ and each cone $Y_j$ we have $|g_i| < \frac18 \systole{Y_j}$.

Let $G\rightarrow \bar G$ denote the quotient $\pi_1X\rightarrow \pi_1X^*$.
Then:
\begin{enumerate}
\item \label{Double:1} $[$Double Coset Separation$]$ \
$\overline{\pi_1 A_1} \bar g_i \overline{\pi_1 A_2} \ \neq \  \overline{\pi_1 A_1} \bar g_j \overline{\pi_1 A_2}$
\ for $i\neq j$.
\end{enumerate}

Suppose moreover that $\{\pi_1A_1g_i\pi_1A_2\}$ form a complete set of double cosets with the property that $\pi_1A_1^{g_i} \cap \pi_1A_2$ is infinite.
Then:

\begin{enumerate}
\setcounter{enumi}{1}
\item \label{Double:2} $[$Square Annular Diagram Replacement$]$ \ If $\overline{\pi_1 A_1}^{\bar g} \cap \overline{\pi_1 A_2}$ is infinite for some $\bar g \in \bar G$,
then $\overline{\pi_1 A_1}\bar g \overline{\pi_1 A_2}= \overline{\pi_1 A_1}\bar g_i \overline{\pi_1 A_2}$ for some (unique)\  $i$.

\item \label{Double:3} $[$Intersections of Images$]$ \
$\overline{\pi_1 A_1^{g_i} \cap \pi_1 A_2}
=  \overline{\pi_1 A_1}^{\bar g_i}\cap \overline{\pi_1 A_2}$
for each $i$.
\end{enumerate}
\end{lem}

\begin{remark}
	\label{rmk:Intersection Control}
The following hold under the assumption of Lemma~\ref{lem:Intersection Control}.
\begin{enumerate}
	\item Let $g\in G$ be arbitrary. If $\overline{\pi_1 A_1}^{\bar g}\cap \overline{\pi_1 A_2}$ is infinite, then there exists $g'$ with $\bar{g}'=\bar{g}$ such that $\overline{\pi_1 A_1^{g'} \cap \pi_1 A_2}=\overline{\pi_1 A_1}^{\bar g}\cap \overline{\pi_1 A_2}$.
	\item Let $g\in G$ be arbitrary. If $(\pi_1 A_1)^{g}\cap \pi_1 A_2$ is infinite
	then $\overline{(\pi_1 A_1)^{g}\cap \pi_1 A_2}=\overline{\pi_1 A_1}^{\bar{g}}\cap\overline{\pi_1 A_2}$.
\end{enumerate}
We only prove (1) as (2) is similar. Note that by Lemma~\ref{lem:Intersection Control} \eqref{Double:2}, $\overline{\pi_1 A_1}\bar g \overline{\pi_1 A_2}= \overline{\pi_1 A_1}\bar g_i \overline{\pi_1 A_2}$. Thus $\bar{g}=\bar{a}_1\bar{g}_i\bar{a}_2$ for some $a_1\in A_1$ and $a_2\in A_2$. Let $g'=a_1g_ia_2$. Then
\begin{align*}
\overline{\pi_1 A_1}^{\bar g}\cap \overline{\pi_1 A_2} &=\overline{\pi_1 A_1}^{\bar{a}_1\bar{g}_i\bar{a}_2}\cap \overline{\pi_1 A_2}=(\overline{\pi_1 A_1}^{\bar g_i}\cap \overline{\pi_1 A_2})^{\bar{a}_2}\\
&=(\overline{\pi_1 A_1^{g_i} \cap \pi_1 A_2})^{\bar{a}_2}=\overline{(\pi_1 A_1^{g_i} \cap \pi_1 A_2)^{a_2}}\\
&=\overline{\pi_1 A_1^{g'} \cap \pi_1 A_2},
\end{align*}
where the third inequality follows from Lemma~\ref{lem:Intersection Control} \eqref{Double:3}.
\end{remark}

\subsection{Superconvexity and fiber products}
\subsubsection{Superconvexity}
The following are quoted from \cite[Def~2.35 \& Lem~2.36]{WiseIsraelHierarchy}:
\begin{definition}
Let $X$ be a metric space.
A subset $Y\subset X$ is \emph{superconvex}
if it is convex and for any bi-infinite geodesic $\gamma$, if $\gamma$ is contained in the $r$-neighborhood $\neb_r(Y)$ for some $r>0$,
then $\gamma\subset Y$.
A map $Y\rightarrow X$ is \emph{superconvex}
if the map $\widetilde Y \rightarrow \widetilde X$ is an embedding onto a superconvex subspace.
\end{definition}

\begin{lem}\label{lem:superconvex core}
Let $H$ be a quasiconvex subgroup of a word-hyperbolic group $G$.
And suppose that $G$ acts properly and cocompactly on
a CAT(0) cube complex $X$.
For each compact subcomplex $D\subset X$
there exists a superconvex $H$-cocompact subcomplex $K\subset X$
such that $D\subset K$.
\end{lem}
The following is a consequence of \cite[Lem~2.39]{WiseIsraelHierarchy}:
\begin{lem}
	\label{lem:bounded wall-piece}
Let $Y\rightarrow X$ be compact and superconvex.
Then there exists $r$ bounding the diameter of every wall-piece in $\langle X\mid Y\rangle$.
 \end{lem}

\subsubsection{Fiber Products}
\label{subsubsec:fiber products}
We record the following  from \cite[Def~8.8 and Lem~8.9]{WiseIsraelHierarchy}:
\begin{definition}[fiber-product]
	\label{def:fiber-product}
Given a pair of combinatorial maps $A\rightarrow X$ and $B\rightarrow X$ between cube complexes,
we define their \emph{fiber-product} $A\otimes_X B$ to be a cube complex,
whose $i$-cubes are pairs of $i$-cubes in $A,B$ that map to the same
$i$-cube in $X$.  There is a commutative diagram:
$$\begin{matrix}
A\otimes_X B & \rightarrow & B \\
\downarrow & & \downarrow \\
A & \rightarrow & X \\
\end{matrix}$$
Note that $A \otimes_X B$ is the subspace of $A\times B$ that is the preimage of the diagonal $D\subset X\times X$
under the map $A\times B\rightarrow X\times X$.
For any cube $Q$, the diagonal of  $Q\times Q$ is isomorphic to  $Q$ by either of the projections,
and this makes $D$ into a cube complex isomorphic to $X$.
We thus obtain an induced cube complex structure on $A\otimes_X B$.

Our description of $A\otimes_X B$ as a subspace of the cartesian product $A\times B$
endows the fiber-product $A\otimes_X B$ with the property of being a universal receiver in the following sense:
Consider a commutative diagram as below.
Then there is an induced map $C\rightarrow A \otimes_X B$ such that the following diagram commutes:
$$\begin{matrix}
C  &  &\longrightarrow &     & B \\
   &\rotatebox{-45}{$\dashrightarrow$} &  & \nearrow &  \\
\downarrow &     & A\otimes_X B &  & \downarrow \\
&\swarrow && \searrow&\\
A &  & \longrightarrow & & X \\
\end{matrix}$$
\end{definition}

\begin{lem}\label{lem:superconvex fiber-product intersection}
Let $A\rightarrow X$ and $B\rightarrow X$ be local isometries of connected nonpositively curved cube complexes.
Suppose the induced lift of universal covers $\widetilde A \subset \widetilde X$ is a superconvex subcomplex.
Then the [noncontractible] components of $A\otimes_X B$ correspond precisely to the [nontrivial] intersections of conjugates of $\pi_1(A,a)$ and $\pi_1(B,b)$ in $\pi_1X$.
\end{lem}

Let $\langle X | \{Y_i\} \rangle$ be a cubical presentation. Then any cone piece of $Y_i$ can be written as the universal cover of some component of $Y_j\otimes_Y Y_i$ in $\widetilde Y_i$.

\section{Stature, depth and big-trees}
\subsection{Big-trees and Stature}
\label{subsec:big-tree}
Let $G$ be the fundamental group of a finite graph of groups with underlying graph $\mathcal{G}$, and let $T$ be the associated Bass-Serre tree. 
A subtree of $T$ is \emph{nontrivial} if it contains at least one edge.
Note that for any nontrivial subtree $S\subset T$, the pointwise stabilizer of $S$, denoted by $\pstab(S)$, is the intersection of the pointwise stabilizers of edges in $S$.
Consequently, $\pstab(S)$ equals the intersection of conjugates of edge groups of $G$
that correspond to the eges of $S$. 

\begin{definition}
	\label{def:big-tree}
A \emph{big-tree} is a nontrivial subtree $S\subset T$ such that
\begin{itemize}
	\item $\pstab(S)$ is infinite;
	\item there does not exist
	a subtree $S'\subset T$ with $S\subsetneq S'$ and $\pstab(S)=\pstab(S')$.
\end{itemize}
\end{definition}

It follows from the definition that $G$ acts on the collection of big-trees of $T$. If two big-trees have the same pointwise stabilizer, then they are the same. Consequently, for a big-tree $S\subset T$, the fixed point set of $\pstab(S)$ is exactly $S$. Moreover, for two big-trees $S_1$ and $S_2$, there exists $g\in G$ such that $gS_1=S_2$ if and only if $g\pstab(S_1)g^{-1}=\pstab(S_2)$.

\begin{definition}[based big-trees, transections and transfer isomorphisms]
	\label{def:based big-trees}
Choose a spanning tree in $\mathcal{G}$ and lift this tree to a subtree $T_{\mathcal{G}}\subset T$. This identifies vertex groups of $G$ with stabilizers of vertices in $T_{\mathcal{G}}$. For each vertex $u\in T$, choose  $g_u\in G$ such that $g_u u\in T_{\mathcal{G}}$. Note that $g_u u$ is unique but $g_u$ might not be unique.
	
A \emph{based big-tree} $(S,v)$ consists of a big-tree $S\subset T$ and a vertex $v\in S$. An \emph{$(S,v)$-transection} is a subgroup 
$\pstab(g_v S)\le \stab(g_v v)$, i.e.\
it is a subgroup of $\stab(g_vv)$ of the form $g_v\pstab(S)g^{-1}_v$. Different choices of $g_v$ for $v$ yield different $(S,v)$-transections, however, they are  conjugate within $\stab(g_vv)$. 

For two different vertices $v_1,v_2\in S$, the inclusions $  \stab(v_1)\hookleftarrow\pstab(S)\hookrightarrow \stab(v_2)$ induce an isomorphism between an $(S,v_1)$-transection and an $(S,v_2)$-transection. This is called a \emph{transfer isomorphism}, which is well-defined up to conjugacy in the vertex groups.
\end{definition}

\begin{lem}
	\label{lem:correspondence}
Let $(S_1,v_1)$ and $(S_2,v_2)$ be based big-trees. There exists $g\in G$ such that $g(S_1,v_1)=(S_2,v_2)$ if and only if there exist $g_1,g_2\in G$ and $w\in T_{\mathcal{G}}$ such that $g_1v_1=g_2v_2=w$ and any $(S_1,v_1)$-transection and $(S_2,v_2)$-transection are conjugtate in $\stab(w)$.
\end{lem}

\begin{proof}
For $i\in\{1,2\}$, let $H_i \le \stab(w)$ be an $(S_i,v_i)$-transection.

Suppose $g(S_1,v_1)=(gS_1,gv_1)=(S_2,v_2)$. Pick $g_2\in G$ such that $g_2v_2=w\in T_{\mathcal{G}}$. We assume without loss of generality that $H_i=g_i\pstab(S_i)g_i^{-1}$ for $i=1,2$, where $g_1=g_2g$. Thus $H_1=g_2(g\pstab(S_1) g^{-1})g_2^{-1}=g_2\pstab(S_2) g_2^{-1}=H_2$.

Suppose $H_1$ and $H_2$ are conjugate in $\stab(w)$. We assume without loss of generality that $H_i=g_i\pstab(S_i)g_i^{-1}$. Hence $g_1\pstab(S_1)g^{-1}_1=hg_2\pstab(S_2)g^{-1}_2h^{-1}$ for $h\in \stab(w)$. Let $k=g_1^{-1}hg_2$. Then $kv_2=v_1$ and $kS_2=S_1$.
\end{proof}

\begin{definition}[$\Upsilon$ and $\Upsilon_V$]
	\label{defn:phi}
Consider the action of $G$ on the collection of based big-trees. For each $G$-orbit, we pick a representative $(S,u)$ and consider an $(S,u)$-transection. Let $\Upsilon$ be the collection of all such transections. For a vertex group $V$ of $G$ (i.e.\ $V=\stab(v)$ for some $v\in T_{\mathcal{G}}$ as above), let $\Upsilon_V\subset\Upsilon$ be the sub-collection of $(S,v)$-transections such that $g_vv$ corresponds to the vertex group $V$.
\end{definition}

\begin{definition}
	\label{def:lowest and higher}
	A big-tree $S\subset T$ is \emph{lowest} if it is not properly contained in another big-tree. A subtree $S\subset T$ is \emph{high} if $|\pstab(S)|=\infty$ and $\pstab(S)$ does not contain any pointwise stabilizer of a lowest big-tree as a finite index subgroup. Similarly, a transection in $\Upsilon$ is \emph{lowest} or \emph{high}, if the associated big-tree is lowest or high. 
\end{definition}

\begin{lem}
	\label{lem:control of intersection of conjugates}
	Let $V$ be a vertex group of $G$. Pick two elements $H_1,H_2$ in $\Upsilon_V$.
	\begin{enumerate}
		\item \label{control of intersections:1} Suppose $H_1,H_2$ are lowest. If $H_1\cap H^g_2$ is infinite for some $g\in V$, then $H_1=H_2$ and $H_1=H^g_1$.
		\item \label{control of intersections:2} If $H_1$ is lowest and $H_2$ is arbitrary, then for any $a_1,a_2\in V$, either $H^{a_1}_1\subset H^{a_2}_2$, or $H^{a_1}_1\cap H^{a_2}_2$ is finite.
	\end{enumerate}	
\end{lem}

\begin{proof}
	For $i\in\{1,2\}$, suppose $H_i$ is an $(S_i,v_i)$-transection. Let $v\in T_{\mathcal{G}}$ be the vertex associated with $V$. Suppose $H_i=g_i\pstab(S_i)g^{-1}_i$ for $i\in\{1,2\}$, where $g_iv_i=v$. Now we prove Part~\eqref{control of intersections:1}. Let $S$ be the convex hull of $g_1 S_1$ and $gg_2 S_2$. Then $\pstab(S)=H_1\cap H^g_2$. Suppose $\pstab(S)$ is infinite. Since both $S_1$ and $S_2$ are lowest, we have $S=g_1 S_1=gg_2 S_2$. Let $k=g^{-1}_1gg_2$. Then $S_1=kS_2$ and $v_1=kv_2$. Thus $H_1=H_2$, $(S_1,v_1)=(S_2,v_2)$ and $g_1=g_2$ by our choice of $\Upsilon_V$. Moreover, $H^g_1=\pstab(gg_1S_1)=\pstab(gg_2S_1)=\pstab(g_1 S_1)=H_1$.
	
	We now prove~\eqref{control of intersections:2}. Let $S$ be the convex hull of $a_1g_1 S_1$ and $a_2g_2 S_2$. Then $H^{a_1}_1\cap H^{a_2}_2=\pstab(S)$. If $H^{a_1}_1\cap H^{a_2}_2$ is infinite, since $a_1g_1 S_1$ is lowest, $S=a_1g_1 S_1$. Hence $a_2g_2 S_2\subset a_1g_1 S_1$. Hence $H^{a_1}_1=\pstab(a_1g_1 S_1)\subset\pstab(a_2g_2 S_2)=H^{a_2}_2$.
\end{proof}

\begin{definition}
	\label{def:finite stature1}
Let $G$ act without inversions on a tree $T$. Say $G$ has \emph{finite stature} (relative  to the action $G\acts T$) if the action of $G$ on the collection of based big-trees has finitely many orbits.
\end{definition}

Note that $G$ satisfies the above definition if and only if the collection $\Upsilon$ is finite.

\begin{example}\label{exmp:height:BS}
Let $G=BS(1,2)=\langle a,t | a^t= a^2\rangle$. The action of $G$ on its Bass-Serre tree has finitely many orbits of big-trees, however, the action does not have finitely many orbits of based big-trees. Thus $G$ does not have finite stature. On the other hand, if $G=BS(1,1)=\langle a,t | a^t= a\rangle$, then $G$ has finite stature with respect to its action on the Bass-Serre tree.
\end{example}

\begin{lem}
	\label{lem:finitely many conjugacy classes0}
Let $G$ act without inversions on a tree $T$.
The following are equivalent:
\begin{enumerate}
	\item $G$ has finite stature with respect to the action $G\acts T$.
	\item For each vertex group $V$ of $G$, there are finitely many $V$-conjugacy classes of infinite subgroups of the form $V\cap (\cap_{e\in E}\stab(e))$, where $E$ is a collection of edges in $T$.
	\item $(G,\mathcal{V})$ has finite stature in the sense of Definition~\ref{def:finite stature}, where $\mathcal{V}$ is the collection of vertex groups of $G$.
\end{enumerate}
\end{lem}

Recall that we have identified $V$ with the stabilizer of a vertex $v\in T_{\mathcal{G}}$.

\begin{proof}
Note that an infinite subgroup $H\le V$ is of form $V\cap (\cap_{e\in E}\stab(e))$ if and only if $H$ is an $(S,v)$-transection. Each element in $\Upsilon_V$ is of such form. Moreover, by Lemma~\ref{lem:correspondence}, each $V$-conjugacy class of such subgroups contains exactly one element inside $\Upsilon_V$. Now the equivalence between (1) and (2) follows. The equivalence between (2) and (3) follows directly from definition.
\end{proof}

In general, if $G$ splits as a graph of groups in two different ways, then it is possible that $G$ has finite stature under one splitting, but not the other splitting, see Example~\ref{example:free by Z}. However, when the splitting of $G$ is already clear, we will only write $G$ has finite stature for simplicity.

\subsection{Depth and Stature}
We now explore a notion measuring the maximal length of an increasing sequence of big trees. There are two variations according to whether the pointwise stabilizer of these big trees are commensurable.

\begin{definition}
	\label{def:depth}
	Let $G$ be a group and let $\Lambda=\{H_1,\ldots,H_r\}$ be a collection of subgroups. The \emph{commensurable depth} of $\Lambda$ in $G$, denoted $\delta_c(G,\Lambda)$, is the largest integer $d$, such that there is a strictly increasing chain $L_1 < \cdots < L_{d}$, where each $L_i$ is the intersection of finitely many conjugates of elements of $\Lambda$. If there are arbitrarily long such sequences, then we define $\delta_c(G,\Lambda)=\infty$. We say $\Lambda$ has \emph{finite commensurable depth} in $G$ if $\delta_c(G,\Lambda)<\infty$.
\end{definition}



\begin{definition}
	Let $\Lambda$ and $G$ be as before. The \emph{depth} of $\Lambda$ in $G$, denoted $\delta(G,\Lambda)$, is the largest integer $d$, such that there is a strictly increasing chain $L_1 <\cdots < L_{d}$ satisfying the conditions that $|L_1|=\infty$, each $L_i$ is an intersection of finitely many conjugates of elements of $\Lambda$, and $[L_{i+1}:L_i]=\infty$. If such $d$ does not exist, then we define $\delta(G,\Lambda)=\infty$.
\end{definition}

\begin{example}
	\label{example:BS}
Note that $\delta_c(G,\Lambda)<\infty$ implies $\delta(G,\Lambda)<\infty$, but the converse may not be true. For instance if $G=\langle a,t | a^t= a^2\rangle$, then the sequence
$\langle a\rangle < \langle a^{t^{-1}}\rangle < \langle a^{t^{-2}}\rangle < \cdots$
shows that $H=\{\langle a\rangle\}$ has infinite commensurable depth. However, $\delta(G,\{\langle a\rangle\})=1$.	
\end{example}

In the rest of this subsection, we return to the scenario where $G$ splits as a graph $\mathcal{G}$ of groups. Recall that we have identified vertex groups of $G$ with vertex stabilizers of a subtree $T_{\mathcal{G}}\subset T$. We assume in addition that each vertex group of $G$ is word-hyperbolic, and each edge group is quasiconvex in its associated vertex groups. Let $\mathcal{E}$ be the collection of edge groups of $G$. 

We recall the following fact which is proven in \cite{GMRS98} (see also \cite{HruskaWisePacking}):
\begin{lem}
	\label{lem:finite height}
	Let $\{H_1,\ldots,H_r\}$ be a collection of quasiconvex subgroups of the word-hyperbolic group $G$. Then $\{H_1,\ldots,H_r\}$ has finite height in $G$.
\end{lem}

Note that each big-tree is uniformly locally finite by Lemma \ref{lem:finite height}, thus $G$ has finite stature if and only if both of the following conditions hold:
\begin{enumerate}
	\item There are finitely many $G$-orbits of big-trees in $T$;
	\item $\stab(S)$ acts cocompactly on each big-tree $S$.
\end{enumerate}	

\begin{lem}
	\label{lem:quasiconvex0}
Suppose each vertex group of $G$ is word-hyperbolic, and each edge group is quasiconvex in its vertex groups. Pick a finite subtree $S\subset T$ and a vertex $v\in S$. Then $\pstab(S)< \stab(v)$ is quasiconvex.
\end{lem}

\begin{proof}
	It suffices to show that $\pstab(S)\hookrightarrow \pstab(e)$	is quasiconvex for a particular edge $e\subset S$, since quasiconvexity is transitive via the vertex groups. We induct on the number of edges in $S$. Let $e\subset S$ be containing a leaf. We remove $e$ from $S$ to obtain another tree $S'$. Let $u=e\cap S'$. By induction, $\pstab(S')$ and $\pstab(e)$ are quasiconvex in $\stab(u)$, hence $\pstab(S)=\pstab(S')\cap\pstab(e)$ is quasiconvex in $\stab(u)$. So $\pstab(S)$ is quasiconvex in $\stab(e)$.
\end{proof}

We recall the following standard fact about hyperbolic groups.
\begin{lem}
	\label{lem:finite subgroups}
	Let $G$ be a word-hyperbolic group. Then it contains only finitely many conjugacy classes of finite subgroups.
\end{lem}

\begin{lem}\label{lem:finite depth}
Let $G$ be as in Lemma~\ref{lem:quasiconvex0}. If $G$ has finite stature, then $\delta_c(G,\mathcal{E})<\infty$. Consequently, $\delta(G,\mathcal{E})<\infty$.
\end{lem}

\begin{proof}
Let $L_1 < \cdots < L_{d}$ be the chain in Definition~\ref{def:depth}. Thus there is a sequence of finite trees $T_1\supset\cdots\supset T_d$ such that $L_i=\pstab(T_i)$. Let $V=\stab(v)$ be a vertex group with $v\in T_{\mathcal{G}}$. Without loss of generality, we assume $v\in T_d$, thus each $L_i$ is an intersection of $V$ with an intersection of finitely many conjugates of elements of $\mathcal{E}$. Thus each $L_i$ is an $V$-conjugate of an element of $\Upsilon_V$. Moreover, if $i\neq j$, then $L_i$ and $L_j$ can not be conjugated to the same element of $\Upsilon_V$. This is because each $L_i$ is quasiconvex in a word-hyperbolic group $V$ (by Lemma~\ref{lem:quasiconvex0}) and a quasiconvex subgroup of a word-hyperbolic group is not conjugated to its proper subgroups (by Lemma~\ref{lem:finite height}). Since $\Upsilon_V$ is finite, we have a bound on the number of infinite order elements in the chain. By Lemma~\ref{lem:finite subgroups}, there is an upper bound on the order of any finite subgroup. This bounds the length of any chain of finite subgroups.
\end{proof}

\begin{remark}\
\begin{enumerate}
	\item Even if $G$ satisfies the assumption of Lemma~\ref{lem:quasiconvex0}, the converse of Lemma~\ref{lem:finite depth} is not true. This can be seen by letting $G$ be the free by cyclic group with a non-standard splitting discussed in Example~\ref{example:free by Z}.
	\item Lemma~\ref{lem:finite depth} is not true for more general groups. For example, let $V=\langle a,s | a^s= a^2\rangle$ and let $G=\langle V,t | a^t= a^2\rangle$. Then $G$ has finite stature with respect to $V$, but $\delta_c(G,V)=\infty$.
\end{enumerate}
\end{remark}

Since $\delta_c(G,\mathcal{E})<\infty$, the pointwise stabilizer of any subtree of $T$ can be expressed as an intersection of finitely many conjugates of edges groups. Thus the following two lemmas hold.

\begin{lem}
\label{lem:quasiconvex}
Lemma~\ref{lem:quasiconvex0} holds without the assumption that $S$ is finite. In particular, for any vertex group $V$, each element in $\Upsilon_V$ is quasiconvex.
\end{lem}

\begin{lem}
	\label{lem:finitely many conjugacy classes1}
	Suppose $G$ satisfies the assumption of Lemma~\ref{lem:quasiconvex0}. Then Lemma~\ref{lem:finitely many conjugacy classes0} still holds if we assume the collection $E$ there is finite.
\end{lem}

Lemma~\ref{lem:finitely many conjugacy classes1} and Corollary~\ref{lem:finitely many conjugacy class} below imply the following.
\begin{cor}
Suppose $G$ is hyperbolic and it splits as a graph of groups such that each vertex group is quasiconvex. Then $G$ has finite stature.
\end{cor}

A proof of the following statement can be found in \cite{GMRS98}
or \cite{HruskaWisePacking}.

\begin{lem}
	\label{lem:finitely many double cosets}
	Let $G$ be a hyperbolic group, and $A,B$ be quasiconvex subgroups. There are finitely many double cosets $BgA$ such that $A\cap B^g$ is infinite.
\end{lem}
In other words, if we consider the collection of all subgroups of form $A\cap B^g$ which are infinite, then there are finitely many $A$-conjugacy classes of such subgroups.

The next result follows from Lemma~\ref{lem:finite height}, Lemma~\ref{lem:finitely many double cosets} and Lemma~\ref{lem:finite subgroups}.
\begin{cor}
	\label{lem:finitely many conjugacy class}
	Let $\{H_1,\ldots,H_r\}$ be a collection of quasiconvex subgroups of the word-hyperbolic group $G$. Let $K$ be a quasiconvex subgroup of $G$. Then there are only finitely many $K$-conjugacy classes of subgroups of the form $K\cap(\cap_{k=1}^{n}H^{g_k}_{k_i})$.
\end{cor}



\subsection{Several observations for passing to finite index subgroups}
We need the following lemmas later when we consider torsion-free finite index subgroups of the edge groups. For each $E\in \mathcal{E}$, we choose a finite index normal subgroup $E'\le E$, and let $\mathcal{E}'=\{E'\}_{E\in\mathcal{E}}$. For each subtree $S\subset T$, let  $E'_S$ denote the intersections of conjugates of elements of $\mathcal{E}'$ corresponding to edges of $S$. Note that $E'_S$ is well-defined since 
$E'\le E$ is normal for each $E'\in \mathcal{E}'$.

\begin{lem}
	\label{lem:uniformly bounded index}
Let $G$ be as in Lemma~\ref{lem:quasiconvex0}. Then for any finite collection of edges $\{e_i\}_{i=1}^{n}$ in $T$, the index $[\cap_{i=1}^{n}\pstab(e_i):\cap_{i=1}^{n}E'_{e_i}]$ is uniformly bounded above. Consequently, $\delta_c(G,\mathcal{E}')<\infty$ and $\delta(G,\mathcal{E}')=\delta(G,\mathcal{E})$.
\end{lem}

\begin{proof}
By Lemma~\ref{lem:finite depth}, $\delta_c(G,\mathcal{E})<\infty$. We claim $[\pstab(S):E'_S]$ is uniformly bounded above for any finite subtree $S$. Note that there is a collection of edges $\{e_i\}_{i=1}^m\subset S$ with $m\le \delta_c(G,\mathcal{E})+1$ such that $\pstab(S)=\cap _{i=1}^{m}\pstab(e_i)$. Thus for any edge $e$ in $S$, $\cap _{i=1}^{m}E'_{e_i}\subset \pstab (e)$, hence
\begin{equation}
\label{equation:index}
[\cap_{i=1}^m E'_{e_i}:(\cap_{i=1}^m E'_{e_i})\cap E'_e]\le [\pstab(e):E'_e].
\end{equation}
Since there are only finitely many $G$-orbits of big-trees, there are finitely many isomorphism types of groups of form $\pstab(S)$, each of which is f.g.\ by Lemma~\ref{lem:quasiconvex}. Since each $\cap_{i=1}^m E'_{e_i}$ is a subgroup of $\pstab(S)$ with its index uniformly bounded above, there are finitely many isomorphism types of groups of form $\cap_{i=1}^m E'_{e_i}$, and each of them is finitely generated. Since $E'_{S}=\cap _{e\subset S} E'_e$, by Equation~\eqref{equation:index}, we have $[\pstab(S):E'_S]$ is uniformly bounded above.	

Finally, let $S$ be the convex hull of $\{e_i\}_{i=1}^{n}$. Then $[\cap_{i=1}^{n}\pstab(e_{\lambda}):\cap_{i=1}^{n}E'_{e_{\lambda}}]\le [\pstab(S):E'_{S}]$, which is uniformly bounded above by previous discussion.
\end{proof}

\begin{definition}
	\label{def:commensurator}
	The \emph{commensurator} $\C_G(H)$ of a subgroup $H$ of $G$, is the subgroup consisting of elements $g\in G$ such that $[H: H^g\cap H]<\infty$.
\end{definition}

The following is proven in \cite{KapovichShort96}:
\begin{lem}
	\label{lem:commensurator}
	Let $H$ be a quasiconvex subgroup of a word-hyperbolic group $G$. Then $H$ has finite index in the commensurator of $H$ inside $G$.
\end{lem}

\begin{lem}
\label{lem:finitely many conjugacy classes2}
Let $G$ be as in Lemma~\ref{lem:quasiconvex0}.
The following are equivalent:
\begin{enumerate}
	\item $G$ has finite stature;
	\item for any vertex group $V$ of $G$ and its associated vertex $v\in T_{\mathcal{G}}$, there are finitely many $V$-conjugacy classes of infinite subgroups of $V$ of the form $E'_S$, where $S$ is a finite subtree containing $v$.
\end{enumerate}
\end{lem}

\begin{proof}
$(1)\Rightarrow (2)$ is a consequence of Lemma~\ref{lem:finitely many conjugacy classes1} and Lemma~\ref{lem:uniformly bounded index}. Now we assume (2). Let $U'$ (resp. $U$) be the collection of infinite subgroups of $V$ which are of form $E'_S$ (resp. $\pstab(S)$) for a finite subtree $v\in S\subset T$. Since $[\pstab(S):E'_S]$ is finite, each element of $U'$ is quasiconvex in $V$ by Lemma~\ref{lem:quasiconvex0}. By (2) and Lemma~\ref{lem:commensurator}, each element of $U'$ is finite index in $\C_V(U')$ with index uniformly bounded above. Since $\pstab(S)$ is contained in the commensurator $\C_V(E'_S)$, $[\C_V(E'_S):\pstab(S)]$ is uniformly bounded above. Since there are finitely many $V$-conjugacy classes in $\C(U')$, the same is true for $U$, hence (1) follows by Lemma~\ref{lem:finitely many conjugacy classes1}.
\end{proof}




\subsection{Stature, Depth and Cleanliness}
In this subsection we examine the behavior of stature and commensurable depth for 
certain graphs of free groups.
As in Examples~\ref{exmp:height:BS}~and~\ref{example:BS},
Baumslag-Solitar groups typically fail to have finite stature and fail to have finite commensurable depth.
Irreducible lattices acting on products of trees \cite{Wise96Thesis,BurgerMozes2000}
also have infinite stature.
This failure is typical for general groups that split as a graph of groups where all edge groups are of finite index in the vertex groups:

\begin{prop}\label{prop:NbyF dichotomy}
Let $G$ split as a finite graph of groups where each edge group has finite index in its vertex groups. Then the following are equivalent:
\begin{enumerate}
\item $G$ has finite stature with respect to its vertex groups.
\item $G$ has finite commensurable depth with respect to its vertex groups.
\item  there is a normal subgroup $N\subset G$ such that $N$ is of finite index in each vertex group, and the quotient $G/N$ is a f.g.\ virtually free group.
\end{enumerate}
\end{prop}
It follows that $G$ has a finite index subgroup that is isomorphic to $N\rtimes F$ for some f.g.\ free group $F$. 
\begin{proof}
$(3)\Rightarrow (1)$ and $(2)$: Let $N\subset G$ be the normal subgroup. For any subtree $S\subset T $ of the Bass-Serre tree,
we have  $N\subset \pstab(S)\subset V$.
Hence $\pstab(S)\subset V$ equals one of the finitely many subgroups of the vertex group $V$
containing $N$. Thus (1) and (2) follow. 

$(1)\Rightarrow(2)$: Let $v$ be a vertex in the Bass-Serre tree $T$. Let $\mathcal{C}$ be the collection of subgroups of $\stab(v)$ of form $\pstab(S)$ where $S$ is a finite subtree of $T$ containing $v$. Then each element of $\mathcal{C}$ is a finite index subgroup of $\stab(v)$. Note that finite index subgroups of $\stab(v)$ of different index can not be conjugated (inside $\stab(v)$). Thus finite stature implies that the $[\stab(v):H]$ is uniformly bounded from above for $H\in\mathcal{C}$. Thus (2) follows.

$(2)\Rightarrow(3)$: For any increasing sequence of finite subtrees $S_1\subset S_2\subset \cdots$,
the  sequence $\pstab(S_1)\geq \pstab(S_2)\geq \cdots$  stabilizes after finitely many proper inclusions. Let $N'$ be the smallest  subgroup arising in this way, and observe that $N'$ is of finite index in $V$ where $v$ is a base vertex of $S$ and $V=\pstab(v)$.
We now check that $N'$ is a normal subgroup of $G$. Observe that $gNg^{-1}= \pstab(gS)$.
However $N' = \pstab(S')$ where $S'$ the smallest subtree containing $S\cup gS$.
Thus $N'$ is normal.
It follows that $\pstab(T)=N'$.
Finally, the quotient $G/N'$ acts faithfully on the locally finite tree $T$,
and hence $F=G/N'$ is virtually free.
\end{proof}

Proposition~\ref{prop:NbyF dichotomy} deceptively suggests that determining finite stature might be accessible and interpretable. However, 
we expect that:
\begin{conjecture}
Then there is no algorithm that takes as input a group
$G$ which splits as a finite graph of f.g.\ free groups,
and outputs certification that $G$ has infinite stature with respect to its vertex groups.
\end{conjecture} 
There is an algorithm that computes the stature when it is finite
for the above $G$, since one can repeatedly compute intersections of subgroups in a free group.

Let us turn now to some more restricted classes of graphs of free groups.
Let $G$ split as a finite graph of groups where all vertex and edge groups are f.g.\ free groups.
The splitting is \emph{algebraically clean} if each edge group embeds as a free factor of each of its vertex groups.
The splitting is \emph{geometrically clean} if it arises from a graph of spaces 
where each vertex space is a graph, and edge space is a graph,
and the attaching maps are topological embeddings.
We caution that we do not assume that the inclusion are combinatorial inclusions of graphs.
(This latter possibility holds when $G$ is the fundamental group of a nonpositively curved $\mathcal{V}\mathcal{H}$-complex.)
For instance, every (f.g.\ free)-by-cyclic group $F\rtimes_\phi \integers$ is algebraically clean, but it is geometrically clean when $\phi$ has finite order.
A cyclic HNN extension $G=F*_{u^t=v}$ is algebraically clean precisely when $u$ and $v$ are both primitive, and it is geometrically clean when $u$ and $v$ belong to a common basis.
When $u$ and $v$ are not distinct powers of the same element (i.e. $u=w^m, v=w^n$ and $m\neq \pm n$), there is a finite index subgroup of $G$ that splits as a geometrically clean graph of free groups (this was first explained in \cite{WiseCycHNN}). The inclusions are  combinatorial in the  rare case where $|u|=|v|$.

\begin{prop}
If $G$ has an algebraically clean splitting then $G$ has finite commensurable depth with respect to its vertex groups.

If $G$ has a geometrically clean splitting then $G$ has finite stature with respect to its vertex groups.
\end{prop}

\begin{proof}
Burns-Chau-Solitar observed that if $A\subset F$ and $B\subset F$ are free factors
of the free group $F$, then $A\cap B$ is a free factor of $B$ \cite{BurnsChauSolitar77}.

We begin by verifying the finite commensurable depth.
\newcommand{\rank}{\text{rank}}
Consider a proper chain of big-trees $S_1 \subsetneq S_2 \subsetneq \cdots$  in the Bass-Serre tree $T$ of the splitting.
We claim that  $\rank(\pstab(S_i))>\rank(\pstab(S_{i+1})$ for each $i$,
and consequently the  length of the chain is bounded by the maximal rank of any vertex group.
To see this, note that $S_{i+1}$ has an edge $e$ such that $e$ is not an edge of  $S_i$ but has an endpoint at a vertex $v$ of $S_i$.
Both $\pstab(e)$ and $\pstab(S_i)$ are free factors of $\stab(v)$,
but since $S_i$ is a big-tree, we see that $\pstab(S_i)\cap \pstab(e)\neq \pstab(S_i)$, and hence 
$\pstab(S_{i+1})$ it is a proper free factor of  $\pstab(S_i)$, and so its rank decreases. 

We now prove finite stature in the geometrically clean case.
This holds by verifying that for any vertex $v$ in a big-tree $S$,
the subgroup $\pstab(S)\subset \stab(v)$ corresponds to the fundamental group of
a subgraph of the vertex space $X_v$.
To see this, consider a chain of subtrees $S_1\subsetneq S_2 \subsetneq \cdots$ of the bass-serre tree $T$,  where $S_{i+1}=S_i\cup e_{i+1}$ and $v_{i+1}=S_i\cap e_{i+1}$.
Suppose  $\pstab(S_i)$ corresponds to a subgraph $C$ of the vertex space $X_{v_{i+1}}$
and $D$ corresponds to the subgraph associated to the inclusion of the edge space $X_{e_{i+1}}\subset X_{v_{i+1}}$.
Then $\pstab(S_{i+1})$ corresponds to the subgraph of $C\cap D$.
The result follows since we can view $\pstab(S)$ as the fundamental group of a subgraph
of the vertex space for any vertex of $S$.
\end{proof}

\begin{example}
\label{example:free by Z}
	Let $\phi$ be a pseudo-Anosov automorphism of $F_2$, the free group of two generators $a$ and $b$. Let $G=F\rtimes_{\phi} \mathbb Z$. $G$ has a splitting with the underlying graph being a circle. It is clear that $G$ has finite stature with respect to its standard splitting. However, we can change the graph of groups structure of $G$ by adding a new edge $e$ to its underlying graph along the vertex of the circle, such that the vertex group at the leaf of $e$ and the edge group of $e$ are the subgroup $\langle a\rangle$ in $F_2$. Since $\phi$ is pseudo-Anosov, $\phi^n(\langle a\rangle)$ and $\langle a\rangle$ are not conjugate inside $F_2$ for any $n\neq 0$. Thus there are infinitely many orbits of based big-trees under this new splitting. Hence $G$ does not have finite stature under such splitting.
\end{example}

Example~\ref{example:free by Z} shows that finite stature may not
hold in the algebraically clean case.
When $G$ has an algebraically clean splitting, all its vertex groups and edge group are separable.
This can be proven in various ways (e.g.\ doubling along a vertex group preserves algebraically clean). More fundamentally, residual finiteness holds since there is a natural system of compatible quotients to graphs of finite groups, and one can extend this argument to see that the edge groups are separable. See \cite{WiseCycHNN}.

Note that when $G$ splits as a finite graph of f.g.\ free groups,
 if all edge groups of $G$ are separable, then $G$ has a finite index subgroup that is algebraically clean (this is if and only if). Consequently, $G$ has finite commensurable depth with respect to its edge groups
 provided they are separable.
Example~\ref{example:free by Z} shows that separability of the edge groups
does not imply finite stature.
However, perhaps it is true in general that separability of the edge groups implies finite commensurable depth.


\section{A separability result for graphs of hyperbolic special groups}
In this section we prove the following theorem.

\begin{thm}
	\label{thm:main1}
	Let $G$ split as a graph of groups with finite underlying graph. Suppose the following conditions hold:
	\begin{enumerate}
		\item each vertex group is word-hyperbolic and virtually compact special;
		\item each edge group is quasiconvex in its vertex groups;
		\item $(G,\mathcal{V})$ has finite stature, where $\mathcal{V}$ is the collection of vertex groups of $G$.
	\end{enumerate}
	Then each quasiconvex subgroup of a vertex group of $G$ is separable in $G$. In particular, $G$ is residually finite.
\end{thm}

\subsection{The depth reducing quotient}
\label{subsec_quotient}
Let $G$ split as a finite graph of groups with  underlying graph $\G$. Let $\{V_i\}_{i=1}^{n}$ be the collection of vertex groups of $G$. Suppose there is an edge $E$ between $V_i$ and $V_j$ (it is possible that $i=j$). Then $E$ induces an isomorphism $\alpha_E:E_i\to E_j$ from a subgroup of $V_i$ to a subgroup of $V_j$. Note that $\alpha_E$ is a transfer isomorphism discussed in Definition~\ref{def:based big-trees}.

Define a \emph{quotient of graphs of groups} as follows. Let $\{q_i:V_i\to \bar{V}_i\}_{i=1}^n$ be a collection of quotient maps. They are \emph{compatible} if for any edge $E$ between $V_i$ and $V_j$, we have $\alpha_E(E_i\cap \ker q_i)=E_j\cap \ker q_j$. In this case, $\alpha_E$ descends to an isomorphism $\bar{\alpha}_E:\bar{E}_i\to\bar{E}_j$, where $\bar{E}_i=q_i(E_i)$. Define a new graph of groups with the same underlying graph $\G$, vertex groups the $\bar{V}_i$'s, and isomorphisms $\{\bar \alpha_E\}$
between edge groups. Let $\bar{G}$ be the fundamental group of this new graph of groups. There is an induced quotient homomorphism $G\to \bar{G}$.

\begin{prop}
	\label{prop:quotients}
Let $G$ be the fundamental group of a finite graph of groups and let $\mathcal{V}$ and $\mathcal{E}$ be the collection of vertex groups and edge groups of $G$. Suppose
\begin{enumerate}
	\item each vertex group is word-hyperbolic and virtually compact special;
	\item each edge group is quasiconvex in its vertex groups;
	\item $G$ has finite stature, and $\delta(G,\mathcal{E})>0$.
\end{enumerate}
For each $V\in \mathcal{V}$, let $Q_V\le V$ be a quasiconvex subgroup. Choose a finite index subgroup $V'\le V$ for each $V\in \mathcal{V}$, and choose a finite index subgroup $E'\le E$ for each $E\in\mathcal{E}$. Then there exists a collection of quotient homomorphisms $\{\phi_V:V\to\bar{V}\}_{V\in \mathcal{V}}$ such that
\begin{enumerate}
	\item \label{conclusion10} $\bar{V}=V/\nclose{\{L^V_i\}}$, where each
	$L^V_i$ is a finite index subgroup of a lowest transection of $G$ in  $V$,
	and the collection varies over representatives of all such lowest transections; moreover, each $L^V_i$ can be chosen such that it is contained in a given finite index subgroup of its associated lowest transection;
	\item \label{conclusion20} for each edge group $E\to V$, 
	the kernel $\ker(E\to\bar{V})$ is generated by $V$-conjugates of $\{L^V_i\}$ that are contained in $E$;
	\item \label{conclusion2} the collection $\{\phi_V:V\to\bar{V}\}_{V\in \mathcal{V}}$ is compatible, hence there is a quotient of graphs of groups $\phi:G\to\bar{G}$ as above;
	\item \label{conclusion1} $\bar{V}$ is word-hyperbolic and virtually compact special for each $V$;
	\item \label{conclusion3} each edge group of $\bar{G}$ is quasiconvex in the corresponding vertex groups;
	\item \label{conclusion4} $(\bar{G},\bar{\mathcal{V}})$ has finite stature;
	\item \label{conclusion5} $\delta(\bar{G},\bar{\mathcal{E}})<\delta(G,\mathcal{E})$;
	\item \label{conclusion6} $\ker \phi|_{V}\le V'$ for each $V\in\mathcal{V}$ and $\ker\phi|_{E}\le E'$ for each $E\in\mathcal{E}$;
	\item \label{conclusion7} Each $\bar{Q}_V$ is quasiconvex in $\bar{V}$.
\end{enumerate}
Moreover, let $S\subset T$ be a finite subtree of the Bass-Serre tree of $G$. Then we can assume that the $G$-equivariant map $\phi_T:T\to \bar{T}$ to the Bass-Serre tree of $\bar{G}$ has the property that $\phi_T|_{_S}$ is injective.
\end{prop}



We  now deduce Theorem~\ref{thm:main1} from  Proposition~\ref{prop:quotients}.
 
\begin{proof}[Proof of Theorem~\ref{thm:main1}]
We induce on $\delta(G,\mathcal{E})$ and first look at the case $\delta(G,\mathcal{E})=0$. Then each element of $\mathcal{E}$ is finite. By a covering space argument, $G$ has a finite index subgroup which is a free product of a free group with a collection of groups which are fundamental groups of hyperbolic special cube complexes. Hence $G$ is hyperbolic and virtually compact special. Any quasiconvex subgroup $Q$ of a vertex group of $G$ is quasiconvex in $G$, hence $Q$ is separable.

Now we assume $\delta(G,\mathcal{E})>0$. Let $Q$ be a quasiconvex subgroup of a vertex group $V$. By Proposition~\ref{prop:quotients}, there exists a quotient $\phi:G\to \bar{G}$ which satisfies all the conditions there, in particular, $\delta(\bar{G},\bar{\mathcal{E}})<\delta(G,\mathcal{E})$. Let $\bar{Q}$ be the image of $Q$ under $V\to\bar{V}$. Pick $g\in G- Q$, we claim that it is possible to choose $\phi$ such that $\bar{g}\notin \bar{Q}$. Assuming this claim, we can deduce the theorem as follows. By Proposition~\ref{prop:quotients}.(6), $\bar{Q}$ is quasiconvex in $\bar V$, hence $\bar Q$ is separable in $\bar G$ by induction, and Theorem~\ref{thm:main1} follows. Now we prove the claim. Let $v\subset T$ be the vertex associated with $V$. Suppose $g$ fixes $v$. Since $Q$ is separable in $V$, there is a finite index normal subgroup $\dot V\le V$ such that $g\notin Q\dot V$. By Proposition~\ref{prop:quotients}.(5), we can choose $\phi$ such that $\ker(V\to\bar{V})\le \dot V$, which implies $\bar{g}\notin \bar{Q}$. Suppose $g$ does not fix $v$. Let $S$ be the convex hull of $v$ and $gv$. By the moreover statement of Proposition~\ref{prop:quotients}, we can assume $\phi_T|_S$ is injective. Since $\phi_T$ is $G$-equivariant, $\bar{g}$ does not stabilize $\phi_T(v)$, hence $\bar{g}\notin \bar{V}$, in particular, $\bar{g}\notin \bar{Q}$.
\end{proof}

\begin{remark}
	\label{rmk:super group acting on cube complex}
	We can actually assume each vertex group of $G$ acts geometrically on a $CAT(0)$ cube complex. This is because of \cite[Lem~7.14]{WiseIsraelHierarchy}, which says that if $A$ is word-hyperbolic and has a finite index subgroup that acts properly and cocompactly on a $CAT(0)$ cube complex, then $A$  acts properly and cocompactly on a  $CAT(0)$ cube complex.
\end{remark}

\begin{remark}[Intuition about compatibility]
We describe how a characteristic subgroup of $\pstab(S)$ for each lowest big-tree $S$
yields a compatible collection of subgroups of the vertex groups of the graph of groups $G$.

For each big tree $S$, there is a short exact sequence
$1\rightarrow \pstab(S) \rightarrow \stab(S) \rightarrow K \rightarrow 1$
where $K$ acts faithfully on $S$. A characteristic subgroup $N$ of $\pstab(S)$ (or more generally, a subgroup $N\le  \pstab(S)$ that is invariant under conjugation by $\stab(S)$) yields a collection of conjugacy classes of subgroups in the vertex groups of $G$ regarded as stabilizers of vertex groups of the Bass-Serre tree $T$.
More precisely, for a vertex $v$ of $T$, the translates $gS$ that contain $v$ yield a collection
 of subgroups of the vertex group $G_v$. When $v\in S$, the distinct conjugacy classes of such subgroups in $G_v$ correspond to the $\stab(S)$ orbits of $v$ in $S$. (More generally, for $v\in gS$, the analogous statement
 holds for $g^{-1}\stab(S)g$ orbits.)
So the $\stab(S)$ orbits of $v$ in $S$ correspond to the collection of various subgroups we will quotient by in $\bar v$,
and compatibility merely reflects that if $u,v$ are vertices that are joined by an edge $e$ in $gS$ for some $g\in G$, then the corresponding quotienting subgroups are isomorphic across that edge. Hence $G_{\bar u},G_{\bar v}$
are compatible across $\bar e$.

Different constraints on the various vertex groups (to ensure some separability, or to lie in our chosen subgroups, or to ensure small-cancellation and hence preserve quasiconvexity and actual (instead of plausible) 
compatibility, are collected together from all the vertex groups of $S$, and we choose $N$ to respect these constraints.

From the viewpoint of the base space, we regard $G$ as the fundamental group of a graph of spaces.
For simplicity, let us assume that $K$ is a free group, and let $N\le \pstab(S)$ be a characteristic subgroup.
Then $K\backslash S$ corresponds to a graph of spaces (which the reader should regard as an $X_N$ bundle over
$\bar S = K\backslash S$) that immerses into the graph of spaces for $G$.
Plausible compatibility corresponds to the fact that the immersed graph of spaces is a bundle,
and hence there is a direct isomorphism between the collection of subgroups of vertex groups on each side of an edge group.

In fact, we are describing ``plausible compatibility''  above: Namely that the subgroups we will kill are the same in a vertex group and an edge group.
Compatibility says that killing them in the vertex group, induces the expected result on  edge groups.
It is a separate technical result that plausible compatibility yields compatibility under certain small cancellation hypotheses that we ensure.

A further important point is the use of Proposition~\ref{prop:lift} which shows that any finite tree with infinite point-wise stabilizer in  $\bar T$ is actually the image of a finite tree with infinite point-wise stabilizer in $T$. This ensures that the depth decreases, since we are quotienting by  finite index subgroups  of the point-wise stabilizers of lowest transections.
This explains our interest in quotienting these particular subgroups.
\end{remark}

\subsection{Proof of Proposition~\ref{prop:quotients}}




As in Definition~\ref{def:lowest and higher}, for each vertex group $V$, let $\{L^{V}_{i}\}_{i=1}^{\ell_V}$ be the collection of lowest elements in $\Upsilon_V$. Let $\{H^{V}_{i}\}_{i=1}^{h_V}$ be the collections of high elements in $\Upsilon_V$. The following is a consequence of Lemma~\ref{lem:commensurator} and Lemma~\ref{lem:control of intersection of conjugates}.\eqref{control of intersections:1} (we recall commensurator and almost malnormality from Definition~\ref{def:commensurator} and Definition~\ref{def:almost malnormal}).
\begin{lem}
$\{\C_V(L^{V}_{i})\}_{i=1}^{\ell_V}$ is almost malnormal in $V$.
\end{lem}

We may assume without loss of generality the $V'$ and the $E'$ of Proposition~\ref{prop:quotients} are normal and torsion-free in their ambient groups. We may moreover assume $E'\le V'$ whenever $E$ is an edge group of $V$.

By Remark~\ref{rmk:super group acting on cube complex}, for each vertex group $V$, let $\widetilde{X}_V$ be a $CAT(0)$ cube complex upon which $V$ acts properly and cocompactly. Let $X_{V'}=\widetilde{X}_V/V'$. We do not assume $X_{V'}$ is virtually special (though it is true), since we only need $X_{V'}$ for cubical small cancellation theory.

The following holds by Lemma~\ref{lem:control of intersection of conjugates} and Lemma~\ref{lem:finite subgroups}.
\begin{lem}
	\label{lem:piece control}
	Let $\widetilde{X}^{L_i}_V\subset \widetilde{X}_V$ be a $\C_V(L^{V}_{i})$-cocompact superconvex subcomplex for each lowest subgroup. Let $\widetilde{X}^{H_i}_V\subset \widetilde{X}_V$ be an $H^V_i$-cocompact superconvex subcomplex for each high subgroup. By possibly enlarging $\widetilde{X}^{H_i}_V$, there exists $M>0$ such that for any $g_1,g_2\in V$ and any $i,i'$ we have:
	\begin{enumerate}
		\item \label{piece control1} either $diam(g_1\widetilde{X}^{L_i}_{V}\cap g_2\widetilde{X}^{L_{i'}}_{V})<M$, or $g_1\widetilde{X}^{L_i}_{V}= g_2\widetilde{X}^{L_{i'}}_{V}$;
		\item \label{piece control2} either $diam(g_1\widetilde{X}^{L_i}_{V}\cap g_2\widetilde{X}^{H_{i'}}_{V})<M$, or $g_1\widetilde{X}^{L_i}_{V}\subset g_2\widetilde{X}^{H_{i'}}_{V}$.
	\end{enumerate}	
\end{lem}

We use $\nclose{\cdots}_H$ to denote the normal closure inside a subgroup $H$ of a collection of elements. If $H$ is clear, we will also write $\nclose{\cdots}$. We use $\langle \cdots\rangle$ to denote the subgroup generated by a collection of elements.

\begin{lem}
	\label{lem:choice}
There exists a collection of finite index normal subgroups $\{\dot{L}^{V}_{i}\,\triangleleft\, \C_V(L^{V}_{i})\}_{V\in\mathcal{V}}$ such that the following properties hold for each $V\in\mathcal{V}$, moreover, they hold for any deeper finite index normal subgroups of the $\C_V(L^{V}_{i})$.
\begin{enumerate}
	\item \label{choice1} Each $\dot{L}^{V}_{i}\le V'\cap L^V_{i}$. Moreover, for each edge group $E$ of $V$, any $V$-conjugate of $\dot{L}^{V}_{i}$ is either contained in $E'$ or intersects $E'$ trivially.
	\item \label{choice2} $\bar{V}=V/\nclose{\dot{L}^{V}_{1},\cdots, \dot{L}^V_{\ell_V}}$ is word-hyperbolic and virtually compact special. For a subgroup $J\le V$, we use $\bar{J}$ to denote the image of $J$ under $V\to\bar{V}$.
	\item \label{choice3} Let $E$ be any edge group of $V$ and let $E'_V=V'\cap E$. Then $\ker (E'_V\to \bar V)$ equals the normal closure in $E'_V$ of all $V$-conjugates of $\{\dot{L}^{V}_{i}\}_{i=1}^{\ell_V}$ that are contained in $E'_V$. Moreover, $\overline{E'_V}$ is quasiconvex in $\bar{V}$.
	\item \label{choice3.5} For each subtree $S$ of the Bass-Serre tree $T$, let $E'_S$ be the subgroup defined right before Lemma~\ref{lem:uniformly bounded index}. Then for each finite subtree $S$ containing the vertex $v\in T$ associated with $V$ with $E'_S$ infinite, any $V$-conjugate of $\dot{L}^{V}_{i}$ is either contained in $E'_S$ or intersects $E'_S$ trivially.
	\item \label{choice4} For any pair $S_1$ and $S_2$ of finite subtrees with infinite pointwise stabilizers with $v\in S_1\cap S_2$, Remark~\ref{rmk:Intersection Control} holds for $E'_{S_1}$ and $E'_{S_2}$ under the quotient homomorphism $V'\to\overline{V'}$.
\end{enumerate}
\end{lem}

\begin{proof}
It suffices to prove that for each property, there exists a collection of subgroups satisfying that property, and the property still holds after passing to further finite index subgroups of elements in this collection.

We first ensure property~\eqref{choice1}. Each infinite edge group in $V$ is an element of $\Upsilon_V$ up to conjugacy in $V$. By Lemma~\ref{lem:control of intersection of conjugates}.(2), for any $g\in V$, either $L^V_i\cap (E')^g$ is trivial, or $L^V_i\cap (E')^g$ is of finite index in $L^V_i$ (the index is uniformly bounded above independent of $g$). We choose $\dot L^V_i$ to be contained in these finite index subgroups.

Property~\eqref{choice2} follows from Theorem~\ref{thm:malnormal special quotient}.

Now we look at property~\eqref{choice3}. We can assume $E'_V$ is infinite. By property~\eqref{choice1}, $\ker(V\to \bar V)\le V'$. For each $\dot{L}^{V}_{i}$, the collection of all $V$-conjugates of $\dot{L}^{V}_{i}$ consists of finitely many $V'$-conjugacy classes. We pick a representative from each $V'$-conjugacy class and form the collection $\{\dot{L}^{V}_{i\lambda}\}_{\lambda\in\Lambda_i}$. Then $\ker(V'\to\bar V')=\nclose{\dot{L}^{V}_{i\lambda}:1\le i\le \ell_V,\, \lambda\in\Lambda_i}_{V'}$, where the subscript $V'$ indicates normal closure inside $V'$. Each $\dot{L}^{V}_{i\lambda}$ acts cocompactly on a translate of $\widetilde{X}^{L_i}_V$ (cf. Lemma~\ref{lem:piece control}), which we denote by $\widetilde{X}^{L_{i\lambda}}_V$. Let $X^{L_{i\lambda}}_V=\widetilde{X}^{L_{i\lambda}}_V/\dot{L}^{V}_{i\lambda}$. Each $\dot{L}^{V}_{i\lambda}$ is residually finite since it is a subgroup of a virtually special group. Hence by Lemma~\ref{lem:piece control}.\eqref{piece control1} and Lemma~\ref{lem:bounded wall-piece}, we can pass to finite index subgroup of $\dot{L}^{V}_{i\lambda}$ such that the systole $\systole{X^{L_{i\lambda}}_V}$ is large enough to ensure that the presentation $\langle X_{V'}\mid X^{L_{i\lambda}}_V:1\le i\le \ell_V,\, \lambda\in\Lambda_i \rangle$ is $C'(1/24)$ (note that $\widetilde{X}^{L_{i\lambda}}_V$ is $\C_V(\dot{L}^{V}_{i\lambda})$-invariant and we always choose $\dot{L}^{V}_{i\lambda}$ to be normal in $\C_V(\dot{L}^{V}_{i\lambda})$, then all the cone-pieces correspond to the first situation of Lemma~\ref{lem:piece control}.\eqref{piece control1}).

Note that $E'_V$ acts cocompactly on a translate of $\widetilde{X}^{L_i}_V$ or $\widetilde{X}^{H_i}_V$ (cf. Lemma~\ref{lem:piece control}), which we denoted by $\widetilde{X}^{E'_V}_V$. Let $X^{E'_V}_V=\widetilde{X}^{E'_V}_V/E'_V$. Since $E'_V$ is torsion-free, and $E'$ is contained in $E'_V$ as a finite index subgroup, property~\eqref{choice1} holds with $E'$ replaced by $E'_V$. Thus by Lemma~\ref{lem:superconvex fiber-product intersection}, each component of the fiber product $X^{E'_V}_V\otimes_{X_{V'}} X^{L_{i\lambda}}_V$ is either a copy of $X^{L_{i\lambda}}_V$, or is a contractible complex. By passing to a further finite cover of each $X^{L_{i\lambda}}_V$, we can assume the diameter of contractible components in $X^{E'_V}_V\otimes_{X_{V'}} X^{L_{i\lambda}}_V$ is $\le \frac{1}{2}\systole{X^{L_{i\lambda}}_V}$ (note that contractible components in the fiber product do not change when we pass to a cover of $X^{L_{i\lambda}}_V$). By Lemma 3.56, the map $\langle X^{E'_V}_V\mid X^{E'_V}_V\otimes_{X_{V'}} X^{L_{i\lambda}}_V:1\le i\le \ell_V, \,\lambda\in\Lambda_i\rangle\to \langle X_{V'}\mid X^{L_{i\lambda}}_V:1\le i\le \ell_V, \,\lambda\in\Lambda_i \rangle$ has liftable shells. We have thus determined the appropriate $\dot L^V_i$ such that Lemma~\ref{lem:quasi-isom embedding} ensures that property~\eqref{choice3} holds for the edge group $E$ of $V$. We repeat this argument for each infinite edge group of $V$ (there are only finitely many edge groups) to find the required collection satisfying property~\eqref{choice3}.

For property~\eqref{choice3.5}, we consider the collection of the various $\{E'_S\}$ where $S$ ranges over all finite subtrees with infinite pointwise stabilizers based at $v$. By Lemma~\ref{lem:finitely many conjugacy classes2}, there are finitely many $V$-conjugacy classes of such subgroups. Thus property~\eqref{choice3.5} can be arranged in the same as property~\eqref{choice1}, using Lemma~\ref{lem:control of intersection of conjugates} and Lemma~\ref{lem:uniformly bounded index}.

It remains to arrange property~\eqref{choice4}. Let $\{E'_S\}$ be the collection in the previous paragraph. By our initial assumption that $E'\le V'$ whenever $E$ is an edge group of $V$, we have $E'_S\le V'$. Since $\{E'_S\}$ has only finitely many $V$-conjugacy classes, it has finitely many $V'$-conjugacy classes. Choose a representative from each $V'$-conjugacy class and form a collection $\{K_i\}_{i=1}^n$. Each $K_i$ is of finite index in a $V$-conjugate of an element in $\Upsilon_V$, thus each $K_i$ acts cocompactly on some translate of $\widetilde{X}^{L_i}_V$ or $\widetilde{X}^{H_i}_V$, which we shall denote by $\widetilde{X}^{K_i}_V$. Let $X^{K_i}_V=\widetilde{X}^{K_i}_V/K_i$.

Since each $K_i$ is quasiconvex in $V'$ (Lemma~\ref{lem:quasiconvex}), there are finitely many double cosets $K_i g K_j$ such that $K^g_i\cap K_j$ is infinite (Lemma~\ref{lem:finitely many double cosets}). Let $\{K_i g^{ij}_\ell K_j\}_{\ell=1}^{\ell_{ij}}$ be the collection of such double cosets. We pass to further finite sheet cover of $X^{L_{i\lambda}}_V$ such that for any $i,j,\ell$, we have $M<\frac{1}{8}\systole{X^{L_{i\lambda}}_V}$ and $|g^{ij}_l|<\frac{1}{8}\systole{X^{L_{i\lambda}}_V}$, where $M$ is the constant in Lemma~\ref{lem:piece control}.\eqref{piece control2}. Then property~\eqref{choice4} follows from Lemma~\ref{lem:Intersection Control} (the \textquotedblleft factor through\textquotedblright\ part of the assumption of Lemma~\ref{lem:Intersection Control} follows from property~\eqref{choice3.5}).
\end{proof}

For each vertex group $V$, we have determined in Lemma~\ref{lem:choice},  a collection $\{\dot{L}^{V}_{i}\}_{i=1}^{\ell_V}$ which we  call \emph{$\dot L$-subgroups}. For each $V$ and $i$, we choose a finite index subgroup $\hat{L}^{V}_{i}\le L^V_i$ such that the collection $\{\hat{L}^{V}_{i}\}_{V\in\mathcal{V},1\le i\le \ell_V}$ satisfies 
\begin{enumerate}
	\item $\hat{L}^{V}_{i}\le \dot L^V_i$ for each $V$;
	\item $\{\hat{L}^{V}_{i}\}$ is compatible with the transfer isomorphisms in Definition~\ref{def:based big-trees}, i.e. any transfer isomorphism from $L^V_i$ to $L^{V'}_{i'}$ maps $\hat{L}^{V}_{i}$ to $\hat{L}^{V'}_{i'}$.
\end{enumerate}
Such choice of finite index subgroups can be made in the following way. Define a relation over the set $\{L^V_i\}_{V\in\mathcal{V},1\le i\le \ell_V}$ such that $L^V_i\sim L^{V'}_{i'}$ if there is a transfer isomorphism between them. This is an equivalence relation by Definition~\ref{def:based big-trees}. For each equivalent class, we pick a representative $L^V_i$, and define $\hat{L}^{V}_{i}$ to be a finite index characteristic subgroup of $L^V_{i}$ such that $\hat{L}^{V}_{i}$ is contained in the images of the $\dot L$-subgroups under the transfer isomorphisms whose ranges are $L^V_{i}$. Then we define the finite index subgroups of other elements in the equivalent class to be the image of $\hat{L}^{V}_{i}$ under transfer isomorphisms.

Let $\ddot L^V_i$ be the normal closure of $\hat{L}^{V}_{i}$ in $\C_V(L^{V}_{i})$. Note that $\ddot L^V_i\le \dot L^V_i$, thus $\{\ddot{L}^{V}_{i}\}_{i=1}^{\ell_V}$ also satisfy Lemma \ref{lem:choice}. Now we consider the quotient map $V\to \bar{V}=V/\nclose{\ddot{L}^{V}_{1},\cdots, \ddot{L}^V_{\ell_V}}$. Note that 
$\nclose{\hat{L}^{V}_{1},\cdots, \widehat{L}^V_{\ell_V}} = \nclose{\ddot{L}^{V}_{1},\cdots, \ddot{L}^V_{\ell_V}}$ 
since  $\nclose{\hat{L}^{V}_{i}}_{\C_V(L_i)} \subset \nclose{\hat{L}^{V}_{i}}_V$ for each $i$.
We employ $\{\ddot{L}^{V}_{i}\}$ to facilitate the small-cancellation conditions.



\textbf{Existence of the quotient map}: Pick two vertex groups $V_i$ and $V_j$ such that there is an edge $E$ between them. Let $\alpha_E:E_i\to E_j$ be as in the beginning of Section~\ref{subsec_quotient}. Let $N_i=\ker(V_i\to \bar V_i)$. It suffices to show $\alpha_E(E_i\cap N_i)=E_j\cap N_j$.

Let $V'_i\le V_i$ be our chosen finite index subgroup and let $E'_{V_i}=V'_i\cap E_i$. Then $E_i\cap N_i=V'_i\cap N_i\cap E_i=N_i\cap E'_{V_i}$ since $N_i\le V'_i$. By Lemma \ref{lem:choice}.\eqref{choice3}, $N_i\cap E'_{V_i}$ is generated by all $V_i$-conjugates of $\{\ddot{L}^{V_i}_{k}\}_{k=1}^{l_{V_i}}$ that are contained in $E'_{V_i}$, hence it is generated by a collection of $V_i$-conjugates of $\{\hat{L}^{V_i}_{k}\}_{k=1}^{l_{V_i}}$. However, $\alpha_E$ maps any $V_i$-conjugate of $\hat{L}^{V_i}_{k}$ to an $V_j$-conjugate of $\{\hat{L}^{V_j}_{k}\}_{k=1}^{l_{V_j}}$. Thus $\alpha_E(N_i\cap E_i)=\alpha_E(N_i\cap E'_{V_i})\subset N_j\cap E'_{V_j}=N_j\cap E_j$. Similarly, $\alpha^{-1}_E(N_j\cap E_j)\subset N_i\cap E_i$.

We have verified the compatibility of  $\{\phi_{V_i}:V_i\to\bar{V}_i\}$. Hence by Lemma~\ref{lem:choice}.\eqref{choice2} and \eqref{choice3}, conclusions~\eqref{conclusion1}, \eqref{conclusion2} and \eqref{conclusion3} of Proposition~\ref{prop:quotients} hold. Proposition~\ref{prop:quotients}.\eqref{conclusion10} holds by construction. Moreover, for $V$, $E$, $E'$ and $E'_V$ in Lemma~\ref{lem:choice}, if an $V$-conjugate of $\{\ddot{L}^{V}_{i}\}_{i=1}^{\ell_V}$ is contained in $E'_V$, then it has a finite index subgroup contained in $E'$, hence it is contained in $E'$ by (1). However, $\ker(E\to\bar{E})=\ker(V\to\bar{V})\cap E$ is generated by such conjugates by the discussion above, thus $\ker(E\to\bar{E})\le E'$ and conclusion~\eqref{conclusion6} of Proposition~\ref{prop:quotients} holds. We deduce Proposition~\ref{prop:quotients}.\eqref{conclusion20} in a similar way. Conclusion~\eqref{conclusion7} of Proposition~\ref{prop:quotients} can be arranged in a similar way as the quasiconvexity statement in Lemma~\ref{lem:choice}.\eqref{choice3}.

\textbf{$(\bar{G},\bar{\mathcal{V}})$ has finite stature:} Let $\mathcal{E}'$ be the collection of finite index subgroups of edge groups chosen at the beginning. By Lemma~\ref{lem:uniformly bounded index}, $\delta(G,\mathcal{E}')<\infty$. Let $\bar{\mathcal{E}}$ be the edge groups of $\bar{G}$ (they are quotients of $\mathcal{E}$), and $\bar{\mathcal{E}}'$ be the finite index subgroups of $\bar{\mathcal{E}}$ that are images of elements of $\mathcal{E}'$. Let $\bar{T}$ be the Base-Serre tree of $\bar{G}$. For each nontrivial subtree $\bar{S}\subset\bar{T}$, we define $\bar{E}'_{\bar{S}}$ in the same way as we defined $E'_S$ for a subtree $S\subset T$.

Let $\phi:G\to \bar{G}$ be the quotient map between graphs of groups induced by $\{\phi_{V_i}:V_i\to\bar{V}_i\}$. Recall that each $\phi_{V_i}$ is formed by quotienting relators satisfying the properties of Lemma~\ref{lem:choice}. Let $\phi_T:T\to\bar{T}$ be the $G$-equivariant map between the Bass-Serre trees of $G$ and $\bar G$ induced by $\phi$. 
\begin{lem}
	\label{lem:lift}
Suppose there is a chain of finite nontrivial subtrees $\bar{S}_1\subset\cdots\subset\bar{S}_n$ with $|\bar{E}'_{\bar{S}_n}|=\infty$. Choose a base vertex $\bar{w}$ in $\bar{S}_1$. Then for any $w\in T$ with $\phi_T(w)=\bar{w}$, there is a sequence of subtrees $w\in S_1\subset S_2\subset\cdots\subset S_n$ such that $\phi(E'_{S_i})=\bar{E}'_{\bar{S}_i}$ and $\phi$ maps $S_i$ isomorphically to $\bar{S}_i$ for each $i$.
\end{lem}

\begin{proof}
We only prove the case $n=1$ as the more general case is similar. We will induct on the number of edges in $\bar{S}_1$. The case when $\bar{S}_1$ has only one edge follows from the definition of quotients between graphs of groups. Now we consider more general $\bar{S}_1$. Let $\bar{e}\subset\bar{S}_1$ be an edge containing a valence one vertex of $\bar{S}_1$. We remove $\bar{e}$ from $\bar{S}_1$ to obtain a smaller tree $\bar{S}_0$. We can also assume the base point $\bar{w}$ is inside $\bar{S}_0$. By induction, we can lift $\bar{S}_0$ to $S_0\subset T$ such that $w\in S_0$ and $\phi(E'_{S_0})=\bar{E}'_{\bar{S}_0}$. Let $\bar{v}=\bar{e}\cap\bar{S_0}$, and let $v$ be the lift of $\bar{v}$ in $S_0$. Let $e\subset T$ be a lift of $\bar{e}$ that contains $v$. Up to conjugacy inside $G$, we assume without loss of generality that $\stab(v)$ is a vertex group $V$ of $G$ (recall that we have identified vertex groups of $G$ with stabilizers of vertices in $T_{\mathcal{G}}\subset T$). We can then view $E'_{S_0}$ and $E'_{e}$ as subgroups of $V'$. Since $\bar{E}'_{\bar{e}}\cap \bar{E}'_{\bar{S}_0}$ is infinite, by Lemma~\ref{lem:choice}.\eqref{choice4} and Remark~\ref{rmk:Intersection Control}, there exists $g\in V'$ with $g\in\ker \phi$ such that $(E'_{e})^{g}\cap E'_{S_0}$ is infinite and $\overline{(E'_{e})^{g}\cap E'_{S_0}}=\bar{E}'_{\bar{e}}\cap\bar{E}'_{\bar{S}_0}$. Then $S_1=S_0\cup g^{-1}e$ is the tree as required (note that $\phi_T(g^{-1}e)=\phi(g^{-1})\phi_T(e)=\bar{e}$). Here $S_1$ must be a high subtree of $T$ for otherwise $\phi(E'_{S_1})$ would be finite.
\end{proof}

To see $\bar{G}$ has finite stature, we verify Lemma~\ref{lem:finitely many conjugacy classes2}.(2). By Lemma~\ref{lem:lift}, we can lift each $\bar{E}'_{\bar{S}}\le \bar G$ with finite $\bar{S}$ to $E'_S\le G$ with finite $S$. However, the collection of $V$-conjugacy classes of all such $E'_S$ is finite since $(G,\mathcal{V})$ has finite stature, hence their $\phi$-images have finitely many $\bar{V}$-conjugacy classes.

\textbf{The depth of $\bar{E}$ decreases:} We verify Proposition~\ref{prop:quotients}.\eqref{conclusion5}. It suffices to show that whenever there is a chain $\bar{S}_1\subset\cdots\subset\bar{S}_n$ with each $\bar{S}_i$ finite, $[\pstab(\bar{S}_{i}):\pstab(\bar{S}_{i+1})]=\infty$ and $|\pstab(\bar{S}_n)|=\infty$, then $n\le \delta(G,\mathcal{E})-1$. Note that $[\bar{E}'_{\bar{S}_{i}}:\bar{E}'_{\bar{S}_{i+1}}]=\infty$. Applying Lemma~\ref{lem:lift}, we obtain a chain of high subtrees $S_1\subset\cdots\subset S_n$ such that $[E'_{S_i}:E'_{S_{i+1}}]=\infty$. So $[\pstab(S_i):\pstab(S_{i+1})]=\infty$. Since $S_n$ is a high subtree, there exists a subtree $S_{n+1}$ such that $S_n\subset S_{n+1}$, $[\pstab(S_n):\pstab(S_{n+1})]=\infty$ and $[\pstab(S_{n+1})]=\infty$. It follows that $n+1\le \delta(G,\mathcal{E})$, hence $n\le \delta(G,\mathcal{E})-1$.

\textbf{Injectivity of finite subtrees:} We verify the moreover statement of Proposition~\ref{prop:quotients}. We assume without loss of generality that $S$ is connected. It suffices to show distinct edges of $S$ meeting a vertex $v$ is sent by $\phi_T$ to distinct edges. Let $\{e_i\}_{i=1}^n$ be the collection of edges of $S$ containing $v$ and in the same $\stab(v)$ orbit. Up to conjugacy in $G$, we assume that $\stab(v)$ is a vertex group $V$. Each $e_i$ corresponds to a coset $g_iE$ of an edge group in $V$. Since $E$ is separable in $V$, there exists a finite index normal subgroup $\dot V\le V$ such that $g_ig^{-1}_j \notin E\dot V$ for any $i\neq j$. Thus provided $\ker(V\to\bar{V})\le \dot V$, $\{g_iE\}$ project to distinct cosets in $\bar{V}$. We repeat this process of constructing $\dot V$ for other $\stab(v)$-orbits of edges in $S$ that contains $v$, as well as other vertices in $S$, to obtain a collection of finite index subgroups of certain vertex groups. By Proposition~\ref{prop:quotients}.\eqref{conclusion6}, it is possible to choose the $\phi_V$'s such that their kernels are contained in these finite index subgroups. Then the resulting $\phi$ now has the additional property that it is injective on $S$. This concludes the proof of Proposition~\ref{prop:quotients}.

We record a consequence of the above construction.
\begin{lem}
	\label{lem:used later}
Let $\phi:G\to \bar{G}$ and $\phi_T:T\to\bar{T}$ be as above. Then
\begin{enumerate}
	\item for any subtree $S\subset T$ with $\pstab(S)$ infinite, $\phi(\pstab(S))$ is commensurable to $\pstab(\bar{S})$ for $\bar{S}=\phi_T(S)$.
	\item Pick finite subtree $\bar{S}\subset\bar{T}$ such that $\pstab(\bar{S})$ is infinite and pick a vertex $\bar{w}\subset\bar{S}$. Then for any $w\in T$ with $\phi_T(w)=\bar{w}$, there exists a subtree $S\subset T$ containing $w$ such that $\phi(\pstab(S))$ is commensurable to $\pstab(\bar{S})$ and $S$ is a lift of $\bar{S}$.
	\item Let $v\in T$ be a vertex and let $H_1$ and $H_2$ be two transections in $\stab(v)$. If $|\phi(H_1)\cap \phi(H_2)|=\infty$, then there exists $g\in \ker(\stab(v)\to\stab(\bar v))$ such that $\phi(H^g_1\cap H_2)=\phi(H_1)\cap \phi(H_2)$ up to finite index subgroups.
\end{enumerate}
\end{lem}

\begin{proof}
For (1), by Lemma~\ref{lem:finite depth} and Lemma~\ref{lem:finite subgroups}, we can assume all the subtrees involved are finite. It suffices to prove $\phi(E'_S)=\bar{E}'_{\bar{S}}$. We induct on the number of edges in $S$. Suppose $S$ is a union of two smaller trees $S=S_1\cup S_2$ such that $S_1$ and $S_2$ intersect in a vertex $v\in T$. Since $E'_S=E'_{S_1}\cap E'_{S_2}$ is infinite, by Lemma~\ref{lem:choice}.\eqref{choice4}, $\phi(E'_{S_1}\cap E'_{S_2})=\phi(E'_{S_1})\cap\phi(E'_{S_2})$ (we view $E'_{S_1}$ and $E'_{S_2}$ as subgroups of $\pstab(v)$, and up to conjugation, we can assume $v=v_A$ in Lemma~\ref{lem:choice}.\eqref{choice4}). However, by induction, $\phi(E'_{S_i})=\bar{E}'_{\bar{S}_i}$ for $i=1,2$. Thus $\phi(E'_S)=\bar{E}'_{\bar{S}}$. (2) is a consequence of Lemma~\ref{lem:lift}. (3) follows from Lemma~\ref{lem:choice}.\eqref{choice4} and that any transection of $G$ can be expressed as the pointwise stabilizer of a finite big-tree of $T$ (cf. Lemma~\ref{lem:finite depth}).
\end{proof}

We extract the following observation from the proof of Lemma~\ref{lem:lift}.
\begin{prop}
	\label{prop:lift}
Let $\phi:G\to \bar{G}$ be the quotient map between graphs of groups induced by quotients $\{\phi_V:V\to \bar V\}_{V\in\mathcal{V}}$ between their vertex groups. Let $\phi_T:T\to\bar{T}$ be the $G$-equivariant map between the Bass-Serre trees of $G$ and $\bar G$ induced by $\phi$. 

Suppose  the following holds whenever
 $H_1 \le V$ is a $V$-conjugate of an edge group of $V$ and $H_2\le V$ is a transection in $V$:
 If $\phi_V(H_1)\cap \phi_V(H_2)$ is infinite
 then there exists $g\in V$ such that $\phi_V(H^g_1\cap H_2)=\phi_V (H_1)\cap \phi_V(H_2)$. 

Then for any finite subtrees $\bar{S}_1\subset\bar{S}_2$ in $\bar T$ with 
$|\pstab(\bar S_2)|=\infty$, and for any lift $S_1\subset T$ with $\phi(S_1)=\bar S_1$
and $\phi(\pstab(S_1))=\pstab(\bar S_1)$,
there exists a lift $S_2$ with $S_1\subset S_2$ and $\phi(S_2)=\bar S_2$
and $\phi(\pstab(S_2))=\pstab(\bar S_2)$.

If we weaken the assumption so that there exists $g\in V$ and finite index subgroups $H'_1\le H_1, H'_2\le H_2$ such that $\phi_V((H'_1)^g\cap H'_2)=\phi_V (H'_1)\cap \phi_V(H'_2)$, then the second conclusion becomes $\phi(\pstab(S_i))$ is commensurable with $\pstab(\bar S_i)$ for each $i$.
\end{prop}

It follows that a sequence of finite trees $\bar S_1\subset \bar S_2 \subset \cdots $ 
with each $|\pstab(\bar S_i)|=\infty$ lifts to a 
sequence $ S_1\subset S_2 \subset \cdots $ 
with $\phi$ mapping each $S_i$ isomorphically to $\bar S_i$
and each $\phi( \pstab(S_i))= \pstab(\bar S_i))$.

\section{Relatively Hyperbolic Application}

\subsection{Background on relative hyperbolicity }
We refer to \cite{HruskaRelQC}
for background on relative hyperbolic groups and the equivalence of various definitions of relative quasiconvexity.

\begin{thm}
	\label{thm:induced peripheral}
	Let $G$ be hyperbolic relative to a collection of maximal parabolic subgroups $\mathcal{P}$ and let $H\le G$ be relatively quasiconvex. 
\begin{enumerate}
	\item There are finitely many $H$-conjugacy classes of infinite subgroups of form $H\cap P^g$ with $g\in G$ and $P\in\mathcal{P}$, moreover, if $\mathcal{P}_H$ is a set of representatives of these conjugacy classes then $H$ is hyperbolic relative to $\mathcal{P}_H$.
	\item Given $P\in\mathcal{P}$, there are finitely many $P$-conjugacy classes of infinite subgroups of form $P\cap H^g$ with $g\in G$.
\end{enumerate}
\end{thm}

The collection $\mathcal{P}_H$ is  the \emph{induced peripheral structure} on $H$.

\begin{proof}
The first assertion is \cite[Thm~9.1]{HruskaRelQC}. Note that (1) implies that there are finitely many double cosets of form $HgP$ such that $H\cap P^g$ is infinite. Hence there are finitely many double cosets of form $PgH$ such that $P\cap H^g$ is infinite. Hence assertion (2) follows.
\end{proof}


The following is a variant of \cite[Prop~5.11]{MartinezPedrozaCombinations2009}:
\begin{thm}
	\label{thm:plus}
	Let $G$ be hyperbolic relative to f.g.\ virtually abelian subgroups $\{P_i\}$, and let $Q$ be a relatively quasiconvex subgroup with its induced peripheral structure. Let $\{K_1, \dots, K_m\}$ be a complete collection of representatives of  maximal parabolic subgroups of $Q$. There exists a collection of finite index subgroups $\{\dot P_i\leq P_i\}$ such that the following holds.
	
	Let $\{\ddot P_i\leq \dot P_i\}$ be a collection of subgroups
	such that $\ddot P_i\leq P_i$ is a finite index normal subgroup  for each $i$. Let $Q^+\leq G$ be the subgroup generated by the union of $Q$ and each conjugate of an element in $\{\ddot P_i\}$ having infinite intersection with $Q$. Then 
	\begin{enumerate}
		\item $Q^+$ is a full quasiconvex subgroup of $G$;
		\item \label{amalgamation} $Q^+$ splits as a tree $T_m$ of groups whose vertex groups are $\{Q,A_1,\ldots,A_m\}$, and where $T_m$ is a wedge of $m$~edges.
		\item The central vertex of $T_m$ has vertex group $Q$,
		and for each $i$ there is an edge between $Q$ and $A_i$ whose edge group is $K_i$.
		\item Each $A_i$ has finite index in a maximal parabolic subgroup of $G$;
		and $\{A_1,\ldots,A_m\}$ is the induced peripheral structure of $Q^+$ in $G$.
	\end{enumerate}
\end{thm}

When $Q$ is parabolic,
$\{K_i\}=\{Q\}$ consists of a single element, $T_m$ is a single edge,
and $A_i=\langle Q\cup g\ddot P_ig^{-1}\rangle$ for some $g,i$. So it is a trivial splitting.

\begin{proof}
	We assume $m\ge 1$ otherwise the theorem is trivial. First we claim there is a collection of finite index normal subgroups $\{\dot P_i\}$ of $\{P_i\}$ such that if $Q'$ is the subgroup of $G$ generated $Q$ and any conjugate of an element in $\{\dot P_i\}$ having infinite intersection with $Q$, then $Q'$ satisfies all the requirements of the theorem. To see the claim, suppose first that  $\{K_i\}$ has a single representative of maximal parabolic subgroup, which is $K_1$. Suppose without loss of generality that $K_1=Q\cap P_1$. Since f.g.\ virtually abelian groups are subgroup separable, $K_1$ is separable in $P_1$, so we can find finite index subgroup $ P'_1\le P_1$ satisfying the assumption of \cite[Thm~1.1]{MartinezPedrozaCombinations2009}. Choose $\dot P_1\le P'_1$ which is finite index and normal in $P_1$. Then the subgroup $A_1$ of $P_1$ generated by $K_1$ and $\dot P_1$ also satisfies the assumption of \cite[Thm~1.1]{MartinezPedrozaCombinations2009}. Hence $\langle Q\cup \dot P_1\rangle_G=\langle Q\cup A_1\rangle_G$ is naturally isomorphic to $Q\ast_{Q\cap A_1} A_1=Q\ast_{K_1} A_1$ as in \cite[Thm~1.1]{MartinezPedrozaCombinations2009}. And this subgroup is quasiconvex. The general case of $m\ge 1$ follows by the same argument together with the induction scheme in \cite[Prop~5.11]{MartinezPedrozaCombinations2009}. By Theorem~\ref{thm:induced peripheral}, the induction terminates after finitely many steps.
	
	We denote the vertex group of $Q'$ by $\{Q,A'_1,\ldots,A'_m\}$. Let $\{\ddot P_i\leq \dot P_i\}$ and $Q^+$ be as in the statement of the theorem. Then $Q^+\le Q'$. Suppose without loss of generality that $\ddot P_i\le A'_i$ for $1\le i\le m$ and suppose $A_i$ is generated by $\ddot P_i$ and $K_i$. Then $Q^+=\langle Q\cup \{(\ddot P_i)^g\}_{1\le i\le m,g\in Q}\rangle=\langle Q\cup \{(A_i)^g\}_{1\le i\le m,g\in Q}\rangle=\langle Q\cup \{A_i\}_{i=1}^m\rangle$. Thus Properties (2) and (3) hold. Now we verify (1). We only consider the case where $A_1\subsetneq A'_1$ and $A_i=A'_i$ for $2\le i\le m$. Other cases are similar. Let $Q_1=\langle Q\cup \{A_i\}_{i=2}^m\rangle$. Then $Q^+=A_1\ast_{K_1}Q_1$ and $Q'=A'_1\ast_{K_1}Q_1$. As $Q'$ is hyperbolic relative to f.g.\ virtually abelian subgroups, it suffices to show $Q^+$ is undistorted in $Q'$ \cite[Thm 1.4 and Thm 1.5]{HruskaRelQC}. Since each vertex group and edge group of $Q'$ is quasiconvex (hence undistorted) in $Q'$, we can view $Q'$ as a tree of spaces over its Bass-Serre tree with vertex spaces and edge spaces undistorted. Now $Q^+$ sits inside $Q'$ as a sub-tree of spaces, hence it is undistorted. Property (4) follows from (1) and Theorem~\ref{thm:induced peripheral}.
\end{proof}


Recall that a subgroup $H$ in a relatively hyperbolic group $G$ is \emph{loxodromic}, if $H$ contains an infinite order element $h$ such that $h$ is not contained in a maximal parabolic subgroup of $G$. The following is a slightly more general form of \cite[Cor~8.6]{HruskaWisePacking}, however, its proof is exactly the same as in \cite{HruskaWisePacking}.
\begin{thm} 
	\label{lem:finitely many double cosets1}
	Suppose $G$ is relatively hyperbolic. Let $H_1,H_2$ be two relatively quasiconvex subgroups of $G$. Then there are only finitely many double cosets $H_1gH_2$ such that $H_1\cap gH_2g^{-1}$ is loxodromic.
\end{thm}

The following is a consequence of \cite[Cor~8.6]{HruskaWisePacking}.
\begin{thm}
	\label{lem:finite height1}
	Let $\{H_1,\ldots,H_r\}$ be a collection of relatively quasiconvex subgroups of the word-hyperbolic group $G$. Then $\{H_1,\ldots,H_r\}$ has finite height in $G$.
\end{thm}

Recall that the notion of finite height in the relatively hyperbolic setting is similar to Definition~\ref{def:height}, except we replace ``is infinite'' there by ``is loxodromic''.

The next result follows from Theorem~\ref{lem:finite height1} and Theorem~\ref{lem:finitely many double cosets1}.
\begin{cor}
	\label{lem:finitely many conjugacy class1}
	Let $\{H_1,\ldots,H_r\}$ be a collection of relatively quasiconvex subgroups of a relatively hyperbolic group $G$. Let $K$ be a quasiconvex subgroup of $G$. Then there are only finitely many $K$-conjugacy classes of loxodromic subgroups of form $K\cap(\cap_{k=1}^{n}H^{g_k}_{k_i})$.
\end{cor}

\subsection{Virtually sparse specialness}
We recall the notion of a ``sparse cube complex'' which is a generalization of a compact cube complex
that arises naturally for cubulations of groups that are hyperbolic relative to virtually abelian subgroups.
We refer to \cite[Sec~7.e]{WiseIsraelHierarchy} for more on this topic.
\begin{definition}[Quasiflat]
	\label{def:quasiflat}
	A \emph{quasiflat} $\widetilde F$ is a locally finite CAT(0) cube complex with a
	proper  action
	by a f.g.\ virtually abelian group $P$
	such that there are finitely many $P$-orbits of hyperplanes.
\end{definition}

\begin{definition}[Sparse]\label{def:sparse}
	Let $G$ be hyperbolic relative to f.g.\ virtually abelian groups $\{P_i\}$.
	We say $G$ acts
	\emph{cosparsely}
	on a CAT(0) cube complex $\widetilde X$
	if there is a compact subcomplex $K$, and quasiflats $\widetilde F_i$ with
	$P_i=\stab(\widetilde F_i)$ for each $i$, such that:
	\begin{enumerate}
		\item $\widetilde X = GK\cup \bigcup_i G\widetilde F_i$
		\item for each $i$ we have $(\widetilde F_i\cap GK) \subset P_iK'$ for some compact $K'$.
		\item for $i,j$ and $g\in G$, either
		$\widetilde F_i \cap g \widetilde F_j \subset GK$ or else $i=j$ and $\widetilde F_i = g\widetilde F_j$.
	\end{enumerate}
	It follows that translates of quasiflats are either equal or are coarsely isolated  in the sense that
	$\diameter(g_i\widetilde F_i \cap g_j\widetilde F_j)$ is bounded by some uniform constant.
	
	We will be especially interested in the case when $G$ acts both properly and cosparsely.
	In particular, when $G$ acts freely and cosparsely on $\widetilde X$,
	then the quotient $X= G\backslash \widetilde X$ is 
	\emph{sparse}.
	In this case, $X$ is the finite union $K\cup \bigcup_i F_i$ where $K$ is compact and $(F_i\cap F_j)\subset K$ for $i\neq j$,
	and each $F_i$ equals $P_i \backslash \widetilde F_i$ where $P_i$ is a f.g.\ virtually abelian group acting freely on a quasiflat $\widetilde F_i$.
\end{definition}

\begin{remark}
	\label{rmk:compactible}
	Suppose $G=\pi_1 X$ for a sparse cube complex $X$. Suppose $G$ is  hyperbolic relative to f.g.\ virtually abelian subgroups $\{P_1,\ldots,P_n\}$ stabilizing quasiflats $\{\widetilde F_1,\ldots,\widetilde F_n\}$ in $\widetilde X$, and $X$ is sparse relative to $\{F_1,\ldots,F_n\}$ with $F_i=\widetilde F_i/P_i$. Now we assume $G$ is hyperbolic relative to another collection of virtually abelian subgroups $\{P'_1,\ldots,P'_m\}$. If $P'_i$ is conjugate to one of $\{P_i\}$, then it stabilizes a translate of an element in $\{\widetilde F_i\}$, which we denote by $\widetilde F'_i$. If $P'_i$ is not conjugate to one of $\{P_i\}$, then $P'_i$ is virtually $\mathbb Z$, and $P'_i$ acts cocompactly on a superconvex subcomplex $\widetilde F'_i$. Thus $X$ is also sparse relative to $\{F'_1,\ldots,F'_n\}$ with $F'_i=\widetilde F'_i/P'_i$. Thus we can always assume the sparse structure of the cube complex is compatible with a given collection of virtually abelian subgroups that $G$ is relative hyperbolic to.
\end{remark}

We refer to \cite[Thm~7.2]{SageevWiseCores} for the following.
A slightly weaker statement is expressed there, but the proof gives the following:
\begin{lem}\label{lem:core}
	Let $G$ be hyperbolic relative to virtually abelian groups, and suppose $G$ acts
	cosparsely
	on a CAT(0) cube complex $\widetilde X$.
	Let $J$ be a full relatively quasiconvex subgroup of $G$, and let $\widetilde K_o$ be a compact subcomplex of $\widetilde X$.
	Then there exists a convex subcomplex $\widetilde Y\subset \widetilde X$ such that $\widetilde K_o\subset \widetilde Y$
	and such that $J$ acts cosparsely on $\widetilde Y$.
	Moreover, 
	let $\{P_i\}$ be the parabolic subgroups of $G$ with $|J\cap P_i|=\infty$,
	and let $\{\widetilde F_i\}$ be the corresponding quasiflats of $G$,
	we may assume that
	$$(J\widetilde K_o \cup_i  J \widetilde F_i)
	\subset
	\widetilde Y
	\subset
	\neb_r(J\widetilde K_o \cup_i  J \widetilde F_i).$$
	
\end{lem}

Note that there might be large parts of $\widetilde Y$ that are in $GK$ but not coarsely in $J$.

It follows that for each $m\geq 0$ there is a uniform upper bound on
$\diameter(\widetilde Y \cap \neb_m(\widetilde F_k))$ unless $\widetilde F_k\subset \widetilde Y$.
Indeed, if the latter statement does not hold, then
$g_i\widetilde F_i \cap \neb_m(\widetilde F_k)$ lies in a finite neighborhood of $GK$,
and hence in a finite neighborhood of $JK_o$. (Here $K$ is the compact subcomplex such that $GK$
contains the intersection of
distinct $G$-translates of the various $\widetilde F_k$.) Thus if $\widetilde F_k$ has infinite
coarse intersection with $\widetilde Y$ then it has infinite coarse intersection with $J\widetilde x$,
and so $|\stab_J(\widetilde F_k)|=\infty$ and hence $\widetilde F_k$ is one of the quasiflats
included in $\widetilde Y$.


The following is proven in \cite[Thm~15.6]{WiseIsraelHierarchy}
\begin{lem}	\label{lem:filling}
	Let $X$ be a virtually special nonpositively curved cube complex that is sparse.
	Suppose $\pi_1X$ is hyperbolic relative to subgroups $P_1,\dots, P_r$
	stabilizing quasiflats $\widetilde F_1,\dots, \widetilde F_r$ of $\widetilde X$, where $X$ is sparse relative to $F_1,\dots, F_r$,
	and each $F_i=P_i\backslash\widetilde F_i$.
	There exist finite index subgroups $P_1^o,\dots, P_r^o$ such that for any
	normal finite index or virtually-cyclic index subgroups
	$P_i^c\subset P_i^o$ the quotient group $G / \nclose{P_1^c, \dots, P_r^c}$ is a word-hyperbolic group
	virtually having a quasiconvex hierarchy terminating in finite groups.
	Hence the quotient group is virtually compact special.
\end{lem}

\begin{cor}
	\label{cor:rhquotient}
	Suppose $G=\pi_1 X$ is hyperbolic relative to a collection of virtually abelian subgroups $\{P_1,\ldots,P_n\}$, where $X$ is a virtually special sparse cube complex. Suppose $Q\le G$ is relatively quasiconvex. Then there exists $\{P'_1,\ldots,P'_n\}$ with each $P'_i$ being finite index in $P_i$ such that for any $\{\dot P_1,\ldots,\dot P_n\}$ with $\dot P_i\le P'_i$, $[P'_i,\dot P_i]<\infty$ and $\dot P_i\trianglelefteq P_i$, we have
	\begin{enumerate}
		\item $\bar{G}=G/\nclose{\dot P_1,\ldots,\dot P_n}$ is word-hyperbolic and virtually compact special;
		\item the image $\bar Q$ of $Q$ under $G\to \bar G$ is quasiconvex;
		\item $\ker(Q\to \bar G)=\nclose{\{Q\cap (\dot P_i)^g\}_{g\in G}}_Q$.
	\end{enumerate}
\end{cor}

\begin{proof}
	By the above discussion, we assume the sparse structure of $X$ is compatible with $\{P_1,\ldots,P_n\}$. Let $Q^+$ be as in Theorem~\ref{thm:plus} with its representatives of maximal parabolic subgroups denoted by $\{A_1,\ldots,A_m\}$. By Lemma~\ref{lem:filling}, we choose $\{\dot P_i\}$ such that $q:G\to \bar G =G/\nclose{\dot P_1,\ldots,\dot P_n}$ satisfies Conclusion~(1). Moreover, we assume for each $i,j$ and $g\in G$, either $A_j\cap (\dot P_i)^g$ is trivial, or $A_j\cap (\dot P_i)^g=(\dot P_i)^g$. Thus $\{\dot P_i\}$ induces a collection $\{\dot A_1,\ldots,\dot A_m\}$ such that $\dot A_i$ is a finite index normal subgroup of $A_i$. Suppose $Q^+$ stabilizes a superconvex subcomplex $\widetilde Y\subset\widetilde X$ as in Lemma~\ref{lem:core}, and suppose each $P_i$ stabilizes a quasiflat $\widetilde F_i\subset \widetilde X$. Then there exists constant $M$ such that for any $g,g'\in G$ and $i$, either $g\widetilde F_i\subset g'\widetilde Y$, or $g\widetilde F_i\cap g'\widetilde Y$ has diameter $\le M$. Since $Q^+$ is full, we can deduce from a liftable shell argument as before that as long as $\systole{\widetilde F_i/\dot P_i}$ is large enough, we have $\ker(Q^+\to \bar G)=\nclose{\dot A_1,\ldots,\dot A_m}_{Q^+}$, the image $\bar Q^+$ of $Q^+$ under $G\to \bar G$ is quasiconvex, and $\ker(A_i\to \bar G)=\dot A_i$ for each $i$. We deduce from Theorem~\ref{thm:plus}.\eqref{amalgamation} and Lemma~\ref{lem:quotient} below that $\ker (Q\to \bar G)=\nclose{K_1\cap\dot A_1,\ldots,K_m\cap\dot A_m}_Q$, where $\{K_i\}$ is as in Theorem~\ref{thm:plus}. Hence conclusion (3) follows. Lemma~\ref{lem:quotient} also implies that $\bar Q^+$ is $\bar Q$ amalgamated with several f.g. virtually abelian groups along its maximal parabolic subgroups. Then each edge of $\bar Q^+$ is contained in a maximal parabolic subgroup of $\bar Q^+$. Thus $\bar Q$ is quasiconvex in $\bar Q^+$ (\cite[Lem~4.9]{BigdelyWiseAmalgams}) and conclusion (2) is true.
\end{proof}

\begin{lem}
	\label{lem:quotient}
	Let $G=A*_CB$ and let $N\trianglelefteq B$. Suppose the quotient $\bar G=G/\nclose{N}$ satisfies that $\ker(B\to \bar G)=N$. Then $\ker (A\to\bar G)=\nclose{C\cap N}_A$, and $\bar G=\bar A*_{\bar C}\bar B$, where $\bar A=A/\ker(A\to\bar G)$, $\bar{B}=B/N$ and $\bar C=C/\ker(C\to\bar G)$.
\end{lem}

\begin{proof}
	Note that $C\cap \ker (A\to\bar G)=\ker (C\to \bar G)=C\cap \ker(B\to\bar G)=C\cap N$. This together with $\nclose{C\cap N}_A\le \ker (A\to\bar G)$ implies that $C\cap \nclose{C\cap N}_A=C\cap N$. Let $\bar A'=A/\nclose{C\cap N}_A$. Then $A\to\bar A'$ and $B\to\bar B$ are compatible, hence induces a quotient of graphs of groups $q:A*_C B\to \bar A'*_{\bar C}\bar B$. Since $N\le \ker q$, so $q$ is a composition $G \stackrel{q_1}{\to} \bar G \stackrel{q_2}{\to} \bar A'*_{\bar C}\bar B$. Since $q(A)=\bar A'$, we have $A\to A/\ker(A\to\bar G)\to \bar A'$ induced by $q$. Thus $A/\ker(A\to\bar G)= \bar A'$. This gives a homomorphism $h:\bar A'*_{\bar C}\bar B\to \bar G$. Then $h\circ q_2$ is identity since it is identity on $q_1(A)$ and $q_1(B)$. Thus $q_2$ is an isomorphism and the lemma follows.
\end{proof}

We say a group $G$ is \emph{virtually sparse special} if $G$ has a finite index torsion free subgroup $H$ such that $H$ is hyperbolic relative to f.g. virtually abelian and $H$ acts cosparsely on a $CAT(0)$ cube complex $\tilde X$ with the quotient $\tilde X/H$ being special.

The following is proved in \cite[Thm~15.13]{WiseIsraelHierarchy}.
\begin{thm}
	\label{thm:sparse separability}
	Suppose $G$ is virtually sparse special. Then any relatively quasiconvex subgroup of $G$ is separable.
\end{thm}

We now discuss the virtually sparse specialness of certain amalgams.

\begin{prop}
	\label{prop:virtual abelian almalgaration}
	Let $G=E\ast_B A$ where $E$ is virtually sparse special, $A$ is a f.g. virtually abelian group and $B$ satisfies at least one of the following conditions
	\begin{enumerate}
		\item $B$ is a maximal parabolic subgroup of $E$;
		\item $B$ is a maximal virtually cyclic group that is loxodromic in $E$.
	\end{enumerate}
	Then $G$ is virtually sparse special.
\end{prop}

The proof of Proposition~\ref{prop:virtual abelian almalgaration} employs Theorem~\ref{thm:abelian hierachy}, which is a consequence of \cite[Thm~18.15]{WiseIsraelHierarchy}, as well as Lemma~\ref{lem:sparse perfection special}, which is a consequence of \cite[Lem~7.56 and Rmk~7.57]{WiseIsraelHierarchy}. The original statement of \cite[Thm~18.15]{WiseIsraelHierarchy} is under a more general condition called strongly sparse, however, this condition is satisfied for fundamental groups of virtually special compact cube complexes that are hyperbolic relative to abelian subgroups.

\begin{definition}$G$ has an \emph{abelian hierarchy} terminating in groups in a class $\mathcal{C}$ if $G$ belongs to the smallest class of groups $\mathcal{M}$ closed under the following conditions:
\begin{enumerate}
	\item $\mathcal{C}\subset\mathcal{M}$;
	\item if $H=A\ast_C B$ with $C$ being f.g.\ free-abelian and $A,B\subset\mathcal{M}$, then $H\in\mathcal{M}$;
	\item if $H=A*_{C^t=C'}$ with $C$ being f.g.\ free-abelian and $A\in\mathcal{M}$, then $H\in\mathcal{M}$.
\end{enumerate}
\end{definition}

\begin{thm}
	\label{thm:abelian hierachy}
	Suppose that $G$ is hyperbolic relative to free-abelian subgroups, and that $G$ has an abelian hierarchy terminating in groups that are fundamental groups of virtually special compact cube complexes that are hyperbolic relative to abelian subgroups. Then $G$ is the fundamental group of a sparse cube complex $X$ that is virtually special.
\end{thm}

\begin{lem}\label{lem:sparse perfection special}
	Let $X\rightarrow R$ be a local-isometry to a compact special cube complex.
	If $X$ is sparse then there is a local-isometry $X\rightarrow X'$
	where $X'$ is compact and $\pi_1X'$ is hyperbolic relative to abelian subgroups $\{P_i'\}$
	that contain the corresponding parabolic subgroups $\{P_i\}$ of the relatively hyperbolic structure of $\pi_1X$.
	And there is a local-isometry $X'\rightarrow R$ such that  $X\rightarrow R$ factors as $X\rightarrow X'\rightarrow R$.
	
	Moreover, we can assume that 
	$\pi_1X'$ splits over a tree $T_r$,
	whose central vertex group is $\pi_1X$,  whose edge groups are the $P_i$,
	and whose leaf vertex groups $P_i'$ are of the form $P_i\times \integers^{m_i}$ for some $m_i$.
\end{lem}

\begin{proof}[Proof of Proposition~\ref{prop:virtual abelian almalgaration}]
	
	By assumption, $E$ has a finite index normal subgroup $E'$ which is the fundamental group of a sparse special cube complex. Moreover, by Theorem~\ref{thm:sparse separability}, we also assume $E'$ is hyperbolic relative to abelian subgroups. First we create a quotient of graph of groups $q:G=E\ast_B A\to \bar G=\bar E\ast_{\bar B}\bar A$ such that 
	\begin{enumerate}
		\item $\ker(E\to \bar E)\subset E'$;
		\item $\bar E$ is virtually compact special;
		\item both $\ker(A\to \bar A)$ and $\ker(B\to \bar B)$ are free abelian.
	\end{enumerate}

	Let $\{P_1,\ldots,P_n\}$ be representatives of maximal parabolic subgroup of $G$ such that $P_1=B$. Let $\{P'_1,\ldots,P'_n\}$ be as in Corollary~\ref{cor:rhquotient}. Let $\dot A\le A$ be a finite index abelian normal subgroup such that $\dot B=\dot A\cap B\le P'_1$. Let $q_E$ be the quotient map $E\to \bar E=E/\nclose{\dot B,\dot P_2,\ldots,\dot P_n}$, where $\dot P_i\le P'_i$ is a finite index abelian normal subgroup of $P_i$ with $\dot P_i\le E_0$ for $2\le i\le n$. Let $q_A$ be the quotient map $A\to A/\dot A$. Corollary~\ref{cor:rhquotient} (3) implies that $B\cap \ker (q_E)=B\cap \ker (q_A)=\dot B$. Thus $q_A$ and $q_E$ induce the desired quotient $q$ with $\bar B=B/\dot B$.
	
	As $\bar A$ and $\bar B$ are finite, we find finite index $\bar G'\le G'$ splitting as a graph of groups with underlying graph $\mathcal{G}$ such that each edge group and vertex group of $\bar G'$ is either trivial or isomorphic to $\bar E_0$ which is a finite index normal torsion free subgroup of $\bar E$ with $\bar E_0\le q_E(E')$. Let $V_1$ (resp. $V_2$) be the collection of vertices of $\mathcal{G}$ whose vertex groups are trivial (resp. isomorphic to $\bar E_0$). Then $(\mathcal{G},V_1,V_2)$ is bipartite. Let $G'=q^{-1}(\bar G')$. Then $G'$ splits as a graph of groups over $\mathcal{G}$ such that
	\begin{enumerate}
		\item a vertex group of $G'$ is of type I (resp. II) if its associated vertex is in $V_1$ (resp. $V_2$), then each type I vertex group of $G'$ is isomorphic to $A_0=\ker q_A$, each type II vertex group of $G'$ is isomorphic to $E_0=(q_E)^{-1}(\bar E_0)$;
		\item $E_0$ is the fundamental group of a sparse special cube complex and $E_0$ is hyperbolic relative to free abelian subgroups;
		\item each edge group of $G'$ isomorphic to $\ker q_B$ and any edge group in a vertex group of type II is a maximal parabolic subgroup of this vertex group.
	\end{enumerate}
	
	Now define a new graph of groups by enlarging each edge group and vertex group of $G'$ as follows. First we enlarge each type I vertex group of $G'$. Let $E_i$ be one such vertex group. Then we enlarge $E_i$ to $E^+_i$ for $1\le i\le k_2$ such that 
	\begin{enumerate}
		\item $E^+_i$ splits as a tree $T_r$ of groups where $T_r$ is an $r$-star with the central vertex group being $E_i$, each edge group being a peripheral subgroup of $E_i$ and each leaf vertex group being free abelian which contains its vertex group as a direct summand;
		\item $E^+_i$ is the fundamental group of a compact special cube complex;
		\item $E^+_i$ is hyperbolic relative to free abelian subgroups and each leaf vertex group of $E^+_i$ is a maximal parabolic subgroup of $E^+_i$.
	\end{enumerate}
	Such enlargement is possible by Lemma~\ref{lem:sparse perfection special}. Second we enlarge each edge group. Let $B_i$ be an edge group. Then exactly one of its vertex groups is of type II, denoted by $E_j$. Since $E_j$ has a unique maximal parabolic subgroup $P$ containing $B_i$ such that $P=B_i\oplus \mathbb Z^m$, we enlarge $B_i$ to $P$. Last we enlarge each vertex group of type II. Let $A_i$ be one such vertex group and let $\{B_i\}_{i=1}^k$ be its edge groups. Since we already enlarge $B_i$ to $P_i=B_i\oplus B'_i$, we enlarge $A_i$ to $A^+_i=A_i\oplus B'_1\oplus\cdots\oplus B'_k$. The boundary map $B_j\to A_i$ naturally extends to $P_j\to A^+_i$. 
	
	Let $G^+$ be the resulting new graph of groups from the previous paragraph. Note that $G^+$ is hyperbolic relative to the $A^+_i$ by \cite[Thm~A]{BigdelyWiseAmalgams}. Thus $G^+$ is the fundamental group of a sparse cube complex $X$ that is virtually special by Theorem~\ref{thm:abelian hierachy}. Since there is a retraction $G^+\to G'$ by our construction, $G'$ is quasi-isometrically embedded in $G^+$. Then it follows from \cite[Thm~7.2]{SageevWiseCores} that $G'$ acts cosparely on a convex subcomplex of $\widetilde X$, which implies that $G'$ is sparse and virtually special, hence the proposition follows.
\end{proof}

\subsection{Quotienting in the relative hyperbolic setting}
In this subsection we prove the following result.

\begin{thm}
	\label{thm:rel hyperbolic}
Let $G$ be hyperbolic relative to subgroups that are virtually f.g.\ free abelian by $\mathbb Z$. Suppose $G$ splits as a finite graph of groups whose edge groups are relative quasiconvex and whose vertex groups are virtually sparse special. Then each relatively quasiconvex subgroup of each vertex group of $G$ is separable. In particular, $G$ is residually finite.
\end{thm}

The assumption of Theorem~\ref{thm:rel hyperbolic} implies that each intersection of a maximal parabolic subgroup of $G$ with a vertex group of $G$ is virtually f.g.\ abelian.


We need a preparatory fact for the proof of the above theorem, which may be useful for controlling stature in general situation.

\begin{lem}[Full Splitting]
	\label{lem:reduction}
Suppose $G$ admits a splitting as in Theorem~\ref{thm:rel hyperbolic} with its collection of edge groups and vertex groups denoted by $\mathcal{E}$ and $\mathcal{V}$. We claim there is a new splitting of $G$ with the same underlying graph such that
\begin{enumerate}
	\item each vertex/edge group of the old splitting is contained in the corresponding vertex/edge group of the new splitting;
	\item each edge group of the new splitting is quasiconvex and full in its vertex group;
	\item each vertex group of the new splitting is virtually sparse special.
\end{enumerate}
\end{lem}

\begin{proof}
Since each edge group is quasiconvex in $G$, so is each vertex group by \cite[Lem~4.9]{BigdelyWiseAmalgams}. First we describe a basic move. Let $E$ be an edge group and $V_1,V_2$ be its vertex groups (it is possible that $V_1=V_2$). To simplify notation, we use the same letter for both the group and its Eilenberg–MacLane space. Suppose $E$ is not full in $V_1$. Let $P\subset V_1$ be a maximal parabolic subgroup such that $P'=P\cap E$ is infinite and is of infinite index in $P$. By \cite[Thm~1.1]{MartinezPedrozaCombinations2009}, we can find a finite index subgroup $\dot P\le P$ such that $P'\le \dot P$ and the natural homomorphism $\dot P\ast_{P'} E\to V_1$ is injective with its image being relatively quasiconvex in $V_1$ (hence the image is also relatively quasiconvex in $G$). We also assume the moreover statement in \cite[Thm~1.1]{MartinezPedrozaCombinations2009} holds. Now consider the following commutative diagram of groups/Eilenberg–MacLane spaces:
$$\begin{matrix}
K & \rightarrow & \dot P \\
\downarrow & & \downarrow \\
E & \rightarrow & V_1 \\
\end{matrix}$$
where $K=\dot P\cap E$. We enlarge $E$ to $E'=(E\sqcup (K\times [0,1])\sqcup \dot P)/\sim$ with left attaching map being $K\to E$ and the right attaching map being $K\to \dot P$. The old boundary map $E\to V_1$ naturally extends to a new boundary map $E'\to V_1$ realizing the injective homomorphism $\dot P\ast_{\dot P\cap E} E\to V_1$. We also enlarge $V_2$ to $V'_2=(V_2\sqcup (K\times [0,1])\sqcup \dot P)/\sim$ with the left attaching map being $K\to E\to V_2$ and the right attaching map being $K\to \dot P$. The boundary map $E\to V_2$ naturally extends to a new boundary map $E'\to V'_2$, which represents the monomorphism $\dot P\ast_{\dot P\cap E} E\to \dot P\ast_{\dot P\cap V_2} V_2$. 

Now we show $V'_2$ is virtually sparse special. By Remark~\ref{rmk:compactible} and Proposition~\ref{prop:virtual abelian almalgaration}, it suffices to show $\dot P\cap V_2=\dot P\cap E$ is either a maximal parabolic subgroup of $V_2$, or a maximal loxodromic virtually cyclic subgroup. The subgroup $P\subset V_1$ in the previous paragraph can be written as $P=P_G\cap V_1$ where $P_G$ is a maximal parabolic subgroup of $G$. As $\dot P\cap E=P_G\cap E$, it remains to show $P_G\cap E=P_G\cap V_2$. Considering the subgroup $H$ of $P_G$ generated by $V_1\cap P_G$, $E\cap P_G$ and $V_2\cap P_G$ ($H$ has a graph of group of structure with these subgroups being its vertex/edge groups). Since $P_G$ is virtually abelian by cyclic, so is $H$. As $P_G\cap E$ is of infinite index in $P_G\cap V_1$, $P_G\cap E\subsetneq P_G\cap V_2$ would indicate that $H$ acts on a tree (which is its Bass-Serre tree) without any invariant vertices or lines, contradicting that $H$ is virtually abelian by cyclic.

Now we have obtained a new graph of spaces and one readily verifies that all the boundary maps induces monomorphisms on the fundamental groups. The new graph of spaces deformation retracts onto the old one, so its fundamental group remains unchanged. The corresponding new graph of groups satisfies all the conditions in Theorem~\ref{thm:rel hyperbolic}.

Now we show the conclusion of the lemma can be reached after finitely many basic moves. Let $T$ be the Bass-Serre tree of $G$ and let $\mathcal{G}=G/T$. Let $P$ be a maximal parabolic subgroup of $G$. Let $S_P$ be the subtree of $T$ spanned by vertices of $T$ whose stabilizers intersect $P$ in infinite subgroups. Note that $S_P$ is $P$-invariant, and $S_P$ also contains all edges of $T$ whose stabilizers intersect $P$ in infinite subgroups. We also know $S_P/P$ is a finite graph by Theorem~\ref{thm:induced peripheral} (2). Since $P$ is virtually free abelian by cyclic, one of the following hold:
\begin{enumerate}[label=(\alph*)]
	\item $P$ stabilizes a vertex $v\in T$;
	\item there is a $P$-invariant line $\ell\subset T$.
\end{enumerate}
Since $G$ is relatively hyperbolic to virtually abelian subgroups, $S_P/P$ is a tree with finitely many edges in case (a) and $S_P/P$ is $\ell/P$ together with finitely many finite trees attached to it in case (b). 

Suppose case (a) holds. Let $\bar v\in S_P/P$ be the image of the fixed vertex $v\in T$ of $P$. Note that $S_P/P$ can be viewed as a tree of groups whose vertex and edge groups are decreasing as we move away from the base point $\bar v$. Define the \emph{complexity} for $P$ to be the number of vertex groups and edge groups of $S_P/P$ which are of infinite index in $P$. If the complexity is $0$, then there is no need to perform any move, otherwise there is an edge in $e\subset S_P/P$ such that exactly one of its vertex groups is finite index in $P$. Now apply the basic move to obtain a new splitting of $G$ with the associated Bass-Serre tree denoted by $T'$. We define $S'_P\subset T'$ in the same way as $S_P$. There is a natural $G$-equivariant graph morphism $\phi:T\to T'$. 
 
We claim $\phi(S_P)=S_P'$. Indeed, it is immediate that $\phi(S_P)\subset S_P'$  since a vertex with infinite $P$-stabilizer must map to a vertex with infinite $P$-stabilizer. Suppose $\phi(S_P)\subsetneq S_P'$, then there is an edge $e'\subset S'_P$ satisfying $e'\nsubseteq \phi(S_P)$. Then there is an edge $e\subset T$ such that either $\stab(e')=\stab(e)$ or $\stab(e')=\stab(e)_{\ast_B}\dot P_0$ where $B$ is a maximal parabolic subgroup of $\stab(e)$ and $\dot P_0$ is of finite index in some maximal parabolic subgroup of $G$. In the former case we have $P\cap \stab(e')=P\cap \stab(e)$, which implies that $e\subset S_P$ and $\phi(e)=e'$. This leads to a contradiction. In the latter case by the moreover statement in \cite[Thm~1.1]{MartinezPedrozaCombinations2009}, $P\cap \stab(e')$ is either conjugated (in $\stab(e')$) to a subgroup of $\stab(e)$, or conjugated to $\dot P_0$. In either situation there exists $g\in \stab(e')$ such that $|(\stab(e))^g\cap P|=\infty$. Thus $ge\subset S_P$ and $\phi(ge)=e'$, which yields a contradiction again. Thus the claim follows.
 

 
The claim implies that there is a surjective map $S_P/P\to S'_P/P$ and one readily sees that the complexity decreases. If $P$ stabilizes a line in $\ell\subset T$, then let $\ell/P$ be the \emph{core} of $S_P/P$. We choose an edge $e\subset S_P/P$ such that exactly one of its vertex group is commensurable to a vertex group in the core, and run the same argument as before. Thus the $P$-complexity is $0$ after finitely many steps. Then we deal with another maximal parabolic subgroup $P'$ in a different conjugacy class. The argument in the previous paragraph implies that $P$-complexity remains $0$ when we decrease $P'$-complexity. Thus we are done after finitely many basic moves.
\end{proof}

%
%
%


The following is a main ingredient in the proof of Theorem~\ref{thm:rel hyperbolic}.
\begin{prop}
	\label{prop:quotient rel hyperbolic}
Let $G$ be as in Theorem~\ref{thm:rel hyperbolic}. Let $\mathcal{E}$ and $\mathcal{V}$ be the collection of edge groups and vertex groups. Suppose in addition that $E$ is full in $V$ whenever $E$ is an edge group of $V$. For each $V\in\mathcal{V}$ (resp. $E\in\mathcal{E}$), we choose a finite index subgroup $V'\le V$ (resp. $E'\le E$). Let $S\subset T$ be a finite subtree of the Bass-Serre tree $T$ of $G$.

Then for each vertex group $V$, there is a quotient $\phi_V:V\to\bar V$ induced by quotienting finite index subgroups of its parabolic subgroups such that the following conditions hold for the collection $\{\phi_V:V\to\bar V\}_{V\in\mathcal{V}}$:
\begin{enumerate}
	\item \label{quotient} these quotients are compatible, so there is an induced quotient $G\to\bar G$;
	\item \label{special} each $\bar V$ is hyperbolic and virtually compact special, moreover, each edge group of $\bar V$ is quasiconvex in $\bar V$;
	\item \label{height} $\bar G$ has finite stature;
	\item \label{subgroup} for each edge group $E$ and each vertex group $V$, $\ker(E\to\bar E)\le E'$ and $\ker(V\to\bar V)\le V'$;
	\item \label{injective} the induced map $\phi_T:T\to\bar T$ between the Bass-Serre trees is injective when restricted to $S$.
\end{enumerate}
\end{prop}
%
%
%

Now we deduce Theorem~\ref{thm:rel hyperbolic} from Proposition~\ref{prop:quotient rel hyperbolic}.

\begin{proof}[Proof of Theorem~\ref{thm:rel hyperbolic}]
It follows from Lemma~\ref{lem:reduction} and Theorem~\ref{thm:sparse separability} that it suffices to prove Theorem~\ref{thm:rel hyperbolic} in the special case that each edge group is full and quasiconvex. Given Theorem~\ref{thm:sparse separability}, this can be deduced from Proposition~\ref{prop:quotient rel hyperbolic} and Theorem~\ref{thm:main1} in the same way as the proof of Theorem~\ref{thm:main1}.
\end{proof}

The following will help control finite stature in Proposition~\ref{prop:quotient rel hyperbolic} (\ref{height}).

\begin{lem}
	\label{lem:transection full}
Under the assumption of Proposition~\ref{prop:quotient rel hyperbolic}, let  $S\subset T$ be a finite subtree of the Bass-Serre tree. Then $\pstab(S)$ is a full quasiconvex subgroup of $\stab(v)$ for any vertex $v\in S$.
\end{lem}

\begin{proof}
	We induct on the number of edges in $S$. The case when $S$ is a vertex is clear. Let $S=S_1\cup_u e$ where $S_1\cap e=u$ is a vertex and $e$ is an edge and $S_1$ is a tree containing $v$. Let $H_1\le_{f.q}H_2$ denote that $H_1$ is a full relatively quasiconvex subgroup of $H_2$. By induction, $\pstab(S_1)\le_{f.q}\stab(u)$. Moreover, $\pstab(e)\le_{f.q}\stab(u)$ as hypothesized in Proposition~\ref{prop:quotient rel hyperbolic}. Thus $\pstab(S)=\pstab(S_1)\cap\pstab(e)\le_{f.q}\stab(u)$ (cf. \cite[Cor~9.5]{HruskaRelQC}). Hence $\pstab(S)\le_{f.q}\pstab(S_1)$. By induction, $\pstab(S_1)\le_{f.q}\stab(v)$, hence $\pstab(S)\le_{f.q}\stab(v)$. 
\end{proof}

\begin{proof}[Proof of Proposition~\ref{prop:quotient rel hyperbolic}]
Let $\{P^V_{ij}\}$ be representatives of maximal parabolic subgroups of a vertex group $V\in\mathcal{V}$ with the induced peripheral structure from $G$. We choose finite index subgroups $\dot P^V_{ij}\trianglelefteq P^V_{ij}$ such that the conclusions of Corollary~\ref{cor:rhquotient} are satisfied for each edge group $E\le V$ playing the role of $Q\le G$ in Corollary~\ref{cor:rhquotient}.

We first describe the proof in the case where each $V$ is the fundamental group of a sparse cube complex $X_V$. We explain the general case at the end of the proof. Let $\widetilde F_{ij}\subset \widetilde X_V$ be the quasiflat stabilized by $P^V_{ij}$. Then there is a constant $M$ such that for any $g_1,g_2\in V$, either $g_1 \widetilde F_{ij}=g_2\widetilde F_{i'j'}$, or $g_1 \widetilde F_{ij}\cap g_2\widetilde F_{i'j'}$ has diameter $\le M$.

By Corollary~\ref{lem:finitely many conjugacy class1}, $V$ contains finitely many $V$-conjugacy classes of loxodromic transections. Let $\{J_i\}$ be representatives of these conjugacy classes. Each $J_i$ is the pointwise stabilizer of a finite subtree of $T$ by Theorem~\ref{lem:finite height1}. Hence each $J_i$ is full and relatively quasiconvex in $V$ by Lemma~\ref{lem:transection full}. Choose a $J_i$-invariant convex subcomplex $\widetilde X^{J_i}_V$ as in Lemma~\ref{lem:core}. Assume $M$ also satisfies that for any $g_1,g_2\in V$, either $g_1\widetilde F_{ij}\subset g_2\widetilde X^{J_{i'}}_V$, or $g_2 \widetilde X^{J_{i'}}_V\cap g_1\widetilde F_{ij}$ has diameter $\le M$. By Theorem~\ref{lem:finitely many double cosets1}, there are finitely many double cosets $J_i g J_j$ in $V$ such that $J^g_i\cap J_j$ is loxodromic. Let $\{J_i g^{ij}_\ell J_j\}$ be the collection of such double cosets. Choose finite index subgroups $\ddot P^V_{ij}\le \dot P^V_{ij}$ such that 
\begin{enumerate}
	\item for any $i',j',i,j,\ell$, we have $M<\frac{1}{8}\systole{\widetilde F_{ij}/\ddot P^V_{ij}}$ and $|g^{i'j'}_l|<\frac{1}{8}\systole{\widetilde F_{ij}/\ddot P^V_{ij}}$;
	\item for any $i,j,i'$ and $g\in V$, either $(P^V_{ij})^g\le J_{i'}$, or $(P^V_{ij})^g\cap J_{i'}=1$ (this is possible since each $J_{i'}$ is full in $V$).
\end{enumerate}
Then Remark~\ref{rmk:Intersection Control} holds for intersections of conjugates of the $J_i$ if we quotient $V$ by any further finite index subgroups of $\{\ddot P^V_{ij}\}$.

Let $\{P_i\}_{i\in I}$ be representatives of maximal parabolic subgroups of $G$. Since each $P_i$ is residually finite, we choose a finite index normal subgroup $\widehat P_i\trianglelefteq P_i$ such that for each $V\in\mathcal{V}$ and $g\in G$, either $(\widehat P_i)^g\cap V$ is contained in a $V$-conjugate of an element in $\{\ddot P^V_{ij}\}$, or $(\widehat P_i)^g\cap V$ is the identity subgroup.

Let $\{\widehat P^V_{ij}\}$ be the collection of finite index subgroups of $\{P^V_{ij}\}$ induced by $\{\widehat P_i\}$. For each $V$, we consider the quotient $q_V:V\to\bar V$ where: $$\bar V\ =\ V/\nclose{\{\widehat P^V_{ij}\}}_V\ =\ V/\langle\{V\cap (\widehat P_i)^g\}_{g\in G, i\in I}\rangle_V$$
Now Conclusion~\eqref{special} holds by our choice of $\widehat P_i$. For any edge group $E\to V$ we have the following where the first equality follows from the choice of $\widehat P_i$ and Corollary~\ref{cor:rhquotient}.(3) and the second equality is a notational restatement as above and the third holds since $E\le V$.
\begin{align*} 
\ker (E\to\bar V) &=  \nclose{\{E\cap (\widehat P^V_{ij})^g\}_{g\in V}}_E =\langle\{E\cap (V\cap (\widehat P_i)^g)\}_{g\in G, i\in I}\rangle_E\\ 
 &= \langle\{E\cap (\widehat P_i)^g\}_{g\in G, i\in I}\rangle_E
\end{align*}
The compatibility of the $q_V$ now follows, since $\ker (E\to\bar V)$ does not depend on $V$ as is clear in the last term above. Thus Conclusion~\eqref{quotient} holds. 

As Remark~\ref{rmk:Intersection Control} holds for intersections of conjugates of the $J_i$, by Proposition~\ref{prop:lift} each transection of $\bar{G}$ lifts to a transection of $G$,
which is necessarily loxodromic in its vertex group since parabolic subgroups of vertex groups of $G$ have finite images. Moreover, being loxodromic in a vertex group implies being loxodromic in $G$ as vertex groups have induced peripheral structure. Then Conclusion~\eqref{height} follows from Corollary~\ref{lem:finitely many conjugacy class1} and Lemma~\ref{lem:finitely many conjugacy classes2}. Given Theorem~\ref{thm:sparse separability}, the rest of the proof is similar to Proposition~\ref{prop:quotients}.

\newcommand{\VV}{{\mathbb V}}

We now explain the modifications needed to handle the general case. Since virtually cosparse implies cosparse \cite[Lem~7.34]{WiseIsraelHierarchy}, we can assume $V$ acts cosparsely on a $CAT(0)$ cube complex $\widetilde X_V$. Let $\VV\le V$ be a finite index  normal subgroup such that $\VV=\pi_1(X_\VV)$ where $X_\VV$ is a sparse cube complex. We do not use that $X_\VV$ is virtually special; instead the cube complex $X_\VV$ supports the cubical small cancellation theory computations. As in the proof of Proposition~\ref{prop:quotients}, we perform small cancellation computations inside $\VV$, and the equalities between normal closures there induce equalities of normal closures in the whole group.
\end{proof}




\bibliographystyle{alpha}
\bibliography{wise}

\def\cprime{$'$} \def\polhk#1{\setbox0=\hbox{#1}{\ooalign{\hidewidth
  \lower1.5ex\hbox{`}\hidewidth\crcr\unhbox0}}} \def\cprime{$'$}
  \def\cprime{$'$} \def\polhk#1{\setbox0=\hbox{#1}{\ooalign{\hidewidth
  \lower1.5ex\hbox{`}\hidewidth\crcr\unhbox0}}}
\begin{thebibliography}{GMRS98}

\bibitem[BCS77]{BurnsChauSolitar77}
R.~G. Burns, T.~C. Chau, and D.~Solitar.
\newblock On the intersection of free factors of a free group.
\newblock {\em Proc. Amer. Math. Soc.}, 64(1):43--44, 1977.

\bibitem[BM00]{BurgerMozes2000}
Marc Burger and Shahar Mozes.
\newblock Lattices in product of trees.
\newblock {\em Inst. Hautes \'Etudes Sci. Publ. Math.}, (92):151--194 (2001),
  2000.

\bibitem[BW13]{BigdelyWiseAmalgams}
Hadi Bigdely and Daniel~T. Wise.
\newblock Quasiconvexity and relatively hyperbolic groups that split.
\newblock {\em Michigan Math. J.}, 62(2):387--406, 2013.

\bibitem[DM17]{DahmaniMj2016}
Francois Dahmani and Mahan Mj.
\newblock Height, graded relative hyperbolicity and quasiconvexity.
\newblock {\em Journal d'Ecole Polytechnique}, 4:515--556, 2017.

\bibitem[GMRS98]{GMRS98}
Rita Gitik, Mahan Mitra, Eliyahu Rips, and Michah Sageev.
\newblock Widths of subgroups.
\newblock {\em Trans. Amer. Math. Soc.}, 350(1):321--329, 1998.

\bibitem[Hru10]{HruskaRelQC}
G.~Christopher Hruska.
\newblock Relative hyperbolicity and relative quasiconvexity for countable
  groups.
\newblock {\em Algebr. Geom. Topol.}, 10(3):1807--1856, 2010.

\bibitem[HW09]{HruskaWisePacking}
G.~Christopher Hruska and Daniel~T. Wise.
\newblock Packing subgroups in relatively hyperbolic groups.
\newblock {\em Geom. Topol.}, 13(4):1945--1988, 2009.

\bibitem[HW18]{HuangWiseSpecial}
Jingyin Huang and Daniel~T. Wise.
\newblock Virtual specialness of certain graphs of special cube complexes.
\newblock {\em Preprint}, pages 1--27, 2018.

\bibitem[JS79]{JacoShalen79}
William~H. Jaco and Peter~B. Shalen.
\newblock Seifert fibered spaces in {$3$}-manifolds.
\newblock {\em Mem. Amer. Math. Soc.}, 21(220):viii+192, 1979.

\bibitem[KS96]{KapovichShort96}
Ilya Kapovich and Hamish Short.
\newblock Greenberg's theorem for quasiconvex subgroups of word hyperbolic
  groups.
\newblock {\em Canad. J. Math.}, 48(6):1224--1244, 1996.

\bibitem[MP09]{MartinezPedrozaCombinations2009}
Eduardo Mart{\'{\i}}nez-Pedroza.
\newblock Combination of quasiconvex subgroups of relatively hyperbolic groups.
\newblock {\em Groups Geom. Dyn.}, 3(2):317--342, 2009.

\bibitem[SW15]{SageevWiseCores}
Michah Sageev and Daniel~T. Wise.
\newblock Cores for quasiconvex actions.
\newblock {\em Proc. Amer. Math. Soc.}, 143(7):2731--2741, 2015.

\bibitem[Wis]{WiseIsraelHierarchy}
Daniel~T. Wise.
\newblock {\em The structure of groups with a quasiconvex hierarchy}.
\newblock Annals of Mathematics Studies, to appear.

\bibitem[Wis96]{Wise96Thesis}
Daniel~T. Wise.
\newblock {\em Non-positively curved squared complexes, aperiodic tilings, and
  non-residually finite groups}.
\newblock PhD thesis, Princeton University, 1996.

\bibitem[Wis00]{WiseCycHNN}
Daniel~T. Wise.
\newblock Subgroup separability of graphs of free groups with cyclic edge
  groups.
\newblock {\em Q. J. Math.}, 51(1):107--129, 2000.

\end{thebibliography}


%
%
\end{document}